\documentclass[reqno,12pt,letterpaper]{amsart}
\usepackage[proof]{sdmacros}

\DeclareMathOperator{\Bor}{Bor}

\title[Fourier dimension and spectral gaps for hyperbolic surfaces]%
{Fourier dimension and spectral gaps\\
for hyperbolic surfaces}
\author{Jean Bourgain}
\email{bourgain@math.ias.edu}
\address{Institute for Advanced Study, Princeton, NJ 08540}
\author{Semyon Dyatlov}
\email{dyatlov@math.mit.edu}
\address{Department of Mathematics, Massachusetts Institute of Technology,
77 Massachusetts Ave, Cambridge, MA 02139}

\begin{document}

\begin{abstract}
We obtain an essential spectral gap for a convex co-compact hyperbolic surface
$M=\Gamma\backslash\mathbb H^2$ which
depends only on the dimension $\delta$ of the limit set.
More precisely, we show that when
$\delta>0$ there exists $\varepsilon_0=\varepsilon_0(\delta)>0$
such that the Selberg zeta function has only finitely many zeroes $s$
with $\Re s>\delta-\varepsilon_0$.

The proof uses the fractal uncertainty principle approach
developed in Dyatlov--Zahl~\cite{hgap}. The key new component
is a Fourier decay bound for the Patterson--Sullivan measure,
which may be of independent interest. This bound
uses the fact that transformations in the group $\Gamma$ are nonlinear,
together with estimates on exponential sums due to Bourgain~\cite{SumProduct}
which follow from the discretized sum-product theorem in $\mathbb R$.
\end{abstract}

\maketitle

\addtocounter{section}{1}
\addcontentsline{toc}{section}{1. Introduction}

Let $M=\Gamma\backslash\mathbb H^2$ be a (noncompact) convex co-compact hyperbolic
surface. The Selberg zeta function $Z_M(s)$
is a product over the set $\mathcal L_M$ of all primitive closed
geodesics
$$
Z_M(s)=\prod_{\ell\in\mathcal L_M} \prod_{k=0}^\infty \big(1-e^{-(s+k)\ell}\big),\quad
\Re s\gg 1,
$$
and extends meromorphically to $s\in\mathbb C$. Patterson~\cite{Patterson3} and Sullivan~\cite{Sullivan} proved that $Z_M$ has a simple zero at the exponent of convergence
of Poincar\'e series, denoted~$\delta$, and no other zeroes in $\{\Re s\geq \delta\}$.
Naud~\cite{NaudGap}, using the method originating in the work of Dolgopyat~\cite{Dolgopyat}, showed that for $\delta>0$, $Z_M$ has only finitely many zeroes
in $\{\Re s\geq \delta-\varepsilon\}$ for some $\varepsilon>0$ depending on the
surface. (See also Petkov--Stoyanov~\cite{PetkovStoyanov},
Stoyanov~\cite{Stoyanov1}, and Oh--Winter~\cite{OhWinter}.)

Our result removes the dependence of the improvement $\varepsilon$ on the surface:
\begin{theo}
  \label{t:gap}
Let $M$ be a convex co-compact hyperbolic surface with $\delta>0$.
Then there exists $\varepsilon_0>0$ depending only on $\delta$
such that $Z_M(s)$ has only finitely many zeroes in $\{\Re s>\delta-\varepsilon_0\}$.
\end{theo}
\Remarks
1. The proof of Theorem~\ref{t:gap} uses the results of Dyatlov--Zahl~\cite{hgap}
and thus gives a resonance free strip with a polynomial
resolvent bound, see~\cite[(1.3)]{hgap}. In the terminology used in~\cite{hgap},
Theorem~\ref{t:gap} gives an \emph{essential spectral gap}
of size ${1\over 2}-\delta+\varepsilon_0$, improving over the Patterson--Sullivan
gap ${1\over 2}-\delta$.

\noindent 2. The Selberg zeta function $Z_M$ has only finitely
many zeroes in $\{\Re s>{1\over 2}\}$; that is, $M$ has
an essential spectral gap of size 0. Therefore, Theorem~\ref{t:gap}
only gives new information when $\delta\leq {1\over 2}+\tilde\varepsilon$
for a small global constant $\tilde\varepsilon>0$.
In~\cite{fullgap} the authors proved that there exists
$\varepsilon>0$ (depending on the surface $M$)
such that $Z_M$ only has finitely many zeroes in $\{\Re s>{1\over 2}-\varepsilon\}$.
The latter result is only interesting when $\delta\geq {1\over 2}$. Therefore~\cite{fullgap}
and the present paper overlap only when $\delta\approx {1\over 2}$, and in the latter
case the present paper gives a stronger result (since $\varepsilon_0$ depends only on $\delta$).
In view of the methods used in~\cite{fullgap} a higher-dimensional extension
of that result seems difficult at the present. See Figure~\ref{f:fup}.

\begin{figure}
\includegraphics{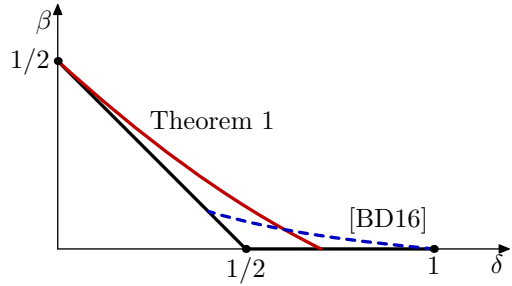}
\caption{The dependence on~$\delta$ of the essential spectral gap $\beta$ (that is, a number
such that $Z_M$ has only finitely many zeroes in $\{\Re s>{1\over 2}-\beta\}$), showing curves representing the bounds of
Theorem~\ref{t:gap} and of~\cite{fullgap}. These curves are for illustration purposes only,
the actual size of the improvement is expected to be much smaller.
The value of $\beta$ from~\cite{fullgap} depends on the surface $M$
but the value given by Theorem~\ref{t:gap} only depends on $\delta$.
The solid black line is the standard (Patterson--Sullivan and Lax--Phillips) gap
$\beta=\max(0,{1\over 2}-\delta)$.}
\label{f:fup}
\end{figure}

\noindent 3. The constant $\varepsilon_1$ can be chosen
increasing in $\delta$, and thus can be made continuous
in~$\delta$~-- see the paragraph preceding~\S\ref{s:higher-dim}.

\noindent 4.
In the more general setting of scattering on manifolds with hyperbolic trapped
sets, the Patterson--Sullivan gap is replaced by the \emph{pressure gap},
established by Ikawa~\cite{Ikawa},
Gaspard--Rice~\cite{GaspardRice}, and Nonnenmacher--Zworski~\cite{NonnenmacherZworskiActa}.
See the reviews of Nonnenmacher~\cite{Nonnenmacher}
and Zworski~\cite{ZworskiReview} for the
history of the spectral gap question and~\cite{hgap,regfup}
for an overview of more recent developments.

\noindent 5. Dyatlov--Jin~\cite{regfup} gave a bound on $\varepsilon_0$
depending only on $\delta$ and the regularity constant
(that is, the constant $C_\Gamma$ in Lemma~\ref{l:ad-regularity}),
proving a fractal uncertainty principle for more general Ahlfors--David
regular sets. Our proof removes the dependence of $\varepsilon_0$
on~$C_\Gamma$
by using the nonlinear nature of the transformations
in the group $\Gamma$. In fact, the earlier work of Dyatlov--Jin~\cite[Proposition~3.17]{oqm}
gives examples of Cantor sets with $\delta\in (0,1/2]$ which
are invariant under a group of linear transformations and do not
satisfy the fractal uncertainty principle we derive for hyperbolic limit sets
here (Propositions~\ref{l:fup-measure} and~\ref{l:fup}).

The key new component of the proof of Theorem~\ref{t:gap},
established in~\S\ref{s:fourier},
is the following generalized Fourier decay bound
for the Patterson--Sullivan measure:
\begin{theo}
  \label{t:fourier}
Let $M,\delta$ be as in Theorem~\ref{t:gap} and denote by $\mu$ the Patterson--Sullivan
measure on the limit set $\Lambda_\Gamma\subset\mathbb R$. 
Assume that
$$
\varphi\in C^2(\mathbb R;\mathbb R),\quad
g\in C^1(\mathbb R;\mathbb C)
$$
are functions satisfying the following bounds for
some constant $C_{\varphi,g}$:
\begin{equation}
  \label{e:fourier-hyp-2}
\|\varphi\|_{C^2}+\|g\|_{C^1}\leq C_{\varphi, g},\quad
\inf_{\Lambda_\Gamma}|\varphi'|\geq C_{\varphi, g}^{-1}.
\end{equation}
Then there exists $\varepsilon_1>0$ depending
only on $\delta$ and there exists $C>0$ depending on $M,C_{\varphi,g}$ such that
\begin{equation}
  \label{e:fourier}
\bigg|\int_{\Lambda_\Gamma} \exp\big(i\xi\varphi(x)\big)g(x)\,d\mu(x)\bigg|\leq
C|\xi|^{-\varepsilon_1}\quad\text{for all }
\xi,\quad
|\xi|>1.
\end{equation}
\end{theo}
\Remarks
1. By taking $\varphi(x)=x$, $g\equiv 1$ on $\Lambda_\Gamma$, we obtain the Fourier
decay bound $\hat\mu(\xi)=\mathcal O(|\xi|^{-\varepsilon_1})$.
This implies that the Fourier dimension $\dim_F\Lambda_\Gamma$
is positive, specifically $\dim_F\Lambda_\Gamma\geq 2\varepsilon_1$.
The nonlinearity of transformations in $\Gamma$ is crucial for obtaining
Fourier decay, since there exist limit sets of linear transformations
(for instance, the mid-third Cantor set) whose Fourier dimension is equal to zero~--
see~\cite[\S12.17]{Mattila}.
Previously Jordan--Sahlsten~\cite{JordanSahlsten}
used a similar nonlinearity property to obtain
Fourier decay for Gibbs measures for the Gauss map which have dimension greater than $1/2$.
(The method of the present paper can be adapted to prove~\cite[Theorem~1.3]{JordanSahlsten}
without the dimensional assumption.)

\noindent 2. The key tool in the proof of Theorem~\ref{t:fourier}
is an estimate on decay of exponential sums established by the
first author~\cite{SumProduct}, see Proposition~\ref{l:bourgain-literally}
and the following remark. In particular our proof relies on the
discretized sum-product theorem for $\mathbb R$.

\noindent 3. The constant $\varepsilon_1$ can be chosen an increasing function of $\delta$.
Indeed, it is determined by the constants $\varepsilon_3,\varepsilon_4,k$
from Proposition~\ref{l:bourgain-literally}, see~\eqref{e:epsilon-1-revealed}
and the proof of Proposition~\ref{l:bourger}. However, Proposition~\ref{l:bourgain-literally}
holds for same $\varepsilon_3,\varepsilon_4,k$ and all larger values of~$\delta$
since the condition~\eqref{e:bl-1} is stronger for larger values
of $\delta_1$ and we apply this proposition with $\delta_1=\delta/24$.

Given Theorem~\ref{t:fourier}, we establish a fractal uncertainty
principle for the limit set $\Lambda_\Gamma$, see Propositions~\ref{l:fup-measure}
and~\ref{l:fup}.
Then Theorem~\ref{t:gap} follows by combining the fractal
uncertainty principle with the results of~\cite{hgap},
see~\S\ref{s:fup}.
The value of $\varepsilon_0$ in Theorem~\ref{t:gap}
can be any number strictly less than $\varepsilon_1/4$,
where $\varepsilon_1$ is obtained in Theorem~\ref{t:fourier},
and thus can be chosen increasing as a function of $\delta$.

\subsection{Extensions to higher dimensional situations}
  \label{s:higher-dim}

While we do not pursue the case of higher-dimensional convex co-compact
hyperbolic quotients in this paper, we briefly discuss a possible generalization of
Theorem~\ref{t:gap} to the case of three-dimensional quotients $M=\Gamma\backslash\mathbb H^3$
with $\Gamma\subset\SL(2,\mathbb C)$ a Kleinian group.

The limit set $\Lambda_\Gamma$ is contained in $\dot{\mathbb C}:=\mathbb C\cup \{\infty\}$
and it is invariant under the action of $\Gamma$ on $\dot{\mathbb C}$
by complex M\"obius transformations. The Patterson--Sullivan measure
is equivariant under $\Gamma$ similarly to~\eqref{e:psin-1}.

Linearizing M\"obius transformations leads to complex multiplication
and the need of a complex analogue of our main tool, Proposition~\ref{l:bourgain-literally}.
In this analogue the measure $\mu_0$ is supported on the annulus
$\{z\in\mathbb C\colon 1/2\leq |z|\leq 2\}$, the
box dimension estimate~\eqref{e:bl-1} is replaced by
\begin{equation}
  \label{e:boxdim-higher}
\sup_{x,\theta\in\mathbb R}
\mu_0\big\{z\colon \Im(e^{i\theta}z)\in [x-\sigma,x+\sigma]\big\}
<\sigma^{\delta_1}
\end{equation}
and the conclusion~\eqref{e:bl-2} is replaced by
$$
\bigg|\int
\exp\big(2\pi i\eta \Im(e^{i\theta} z_1\cdots z_k)\big)\,d\mu_0(z_1)\cdots d\mu_0(z_k)
\bigg|\leq N^{-\varepsilon_4},\quad
\theta\in\mathbb R.
$$
This complex analogue of Proposition~\ref{l:bourgain-literally}
can be shown by following the proof of~\cite[Lemma~8.43]{SumProduct}
and replacing the real version of the sum-product theorem~\cite[Theorem~1]{SumProduct}
by its complex version established in~\cite[Proposition~2]{BourgainGamburd}.

However, the box dimension bound~\eqref{e:boxdim-higher} is more subtle than
in the case of surfaces. Indeed, in the case of a hyperbolic cylinder
(i.e. when $\Gamma$ is a co-compact subgroup of~$\SL(2,\mathbb R)$,
with $\delta=1$)
the limit set $\Lambda_\Gamma$ is equal to $\mathbb R\subset\mathbb C$
and the Patterson--Sullivan measure equals the Poisson measure
$\pi^{-1}(1+x^2)^{-1}\,dx$.
In this case, both~\eqref{e:boxdim-higher} and the Fourier decay bound~\eqref{e:fourier}
fail.

In fact, for hyperbolic cylinders
the specific fractal uncertainty principle~\cite[Definition~1.1]{hgap}
used to establish the spectral gap still holds (and does recover
the correct size of the spectral gap, equal to~${1\over 2}$), however the
general fractal uncertainty principle (Proposition~\ref{l:fup-measure}) fails
if we take the phase function $\Phi(z,w)=\Im(zw)$ which
restricts to 0 on $\Lambda_\Gamma\times\Lambda_\Gamma=\mathbb R^2\subset\mathbb C^2$
but has nondegenerate matrix of mixed derivatives $\partial_{(z,\bar z)}\partial_{(w,\bar w)}\Phi$.

\section{Structure of the limit set}
\label{s:limit-structure}

In this section, we study limit sets of convex co-compact quotients,
as well as the associated group action and Patterson--Sullivan measure,
establishing their properties which form the basis for the proof of the Fourier decay
bound in~\S\ref{s:fourier}.

Let $M=\Gamma\backslash\mathbb H^2$ be a convex co-compact hyperbolic surface.
Here $\mathbb H^2$ is the upper half-plane model of the
hyperbolic plane and $\Gamma$ is a convex co-compact
(in particular, discrete) subgroup of
$\SL(2,\mathbb R)$ acting isometrically on $\mathbb H^2$ by
M\"obius transformations:
$$
\gamma=\begin{pmatrix} a & b \\ c & d \end{pmatrix}\in\SL(2,\mathbb R),\quad
z\in \mathbb H^2=\{z\in\mathbb C\mid \Im z >0\}
\quad\Longrightarrow\quad
\gamma(z)={az+b\over cz+d}.
$$
The action of $\SL(2,\mathbb R)$ extends continuously to the compactified
hyperbolic plane
$$
\overline{\mathbb H^2}:=\mathbb H^2\cup \dot{\mathbb R},\quad
\dot{\mathbb R}:=\mathbb R\cup\{\infty\}.
$$
See for instance the book of Borthwick~\cite[Chapter~2]{BorthwickBook}
for more details.

We assume that $M$ is nonelementary and noncompact
and introduce the following notation:
\begin{itemize}
\item $\delta\in (0,1)$, the exponent of convergence of Poincar\'e series,
see~\cite[\S2.5.2]{BorthwickBook};
\item $\Lambda_\Gamma\subset\dot{\mathbb R}$, the limit set of the group $\Gamma$,
see~\cite[\S2.2.1]{BorthwickBook};
\item $\mu$, the Patterson--Sullivan measure (centered at $i\in\mathbb H^2$)
which is a probability measure on $\dot{\mathbb R}$ supported
on $\Lambda_\Gamma$, see~\cite[\S14.1]{BorthwickBook}.
\end{itemize}

\subsection{Schottky groups}
\label{s:schottky}

A \emph{Schottky group} is a convex co-compact subgroup $\Gamma\subset\SL(2,\mathbb R)$
constructed in the following way (see~\cite[\S15.1]{BorthwickBook}
and Figure~\ref{f:schottky}):
\begin{itemize}
\item Fix nonintersecting closed
half-disks $D_1,\dots,D_{2r}\subset \overline{\mathbb H^2}$ centered
on the real line. Here $r\in\mathbb N$ and for the nonelementary cases studied
here, we have $r\geq 2$.
\item Put $\mathcal A:=\{1,\dots,2r\}$ and for each
$a\in\mathcal A$, denote
$$
\overline{a}:=\begin{cases} a+r,& 1\leq a\leq r;\\
a-r, & r+1\leq a\leq 2r.
\end{cases}
$$
\item Fix transformations $\gamma_1,\dots,\gamma_{2r}\in\SL(2,\mathbb R)$ such that
for all $a\in\mathcal A$,
\begin{equation}
  \label{e:schottky-mapping}
\gamma_a(\overline{\mathbb H^2}\setminus D_{\overline a}^\circ)=D_a,\quad
\gamma_{\overline a}=\gamma_a^{-1}.
\end{equation}
\item Let $\Gamma\subset\SL(2,\mathbb R)$ be the free group generated by $\gamma_1,\dots,\gamma_r$.
\end{itemize}
Each convex co-compact group $\Gamma\subset\SL(2,\mathbb R)$ can be represented
in the above way for some choice of $D_1,\dots,D_{2r}$, $\gamma_1,\dots,\gamma_{2r}$,
see~\cite[Theorem~15.3]{BorthwickBook}.
We henceforth fix a Schottky structure for $\Gamma$.

\noindent\textbf{Notation:}
In the rest of the paper, $C_\Gamma$ denotes constants which only depend
on the Schottky data $D_1,\dots,D_{2r},\gamma_1,\dots,\gamma_{2r}$,
whose exact value may differ in different places.
The elements of $\Gamma$ are indexed by words on the generators
$\gamma_1,\dots,\gamma_{2r}$. We introduce some useful combinatorial notation:
\begin{itemize}
\item For $n\in\mathbb N_0$, define $\mathcal W_n$,
the set of words of length $n$, by
$$
\mathcal W_n:=\{a_1\dots a_n\mid a_1,\dots,a_n\in \mathcal A,\quad
a_{j+1}\neq \overline{a_j}\quad\text{for }j=1,\dots,n-1\}.
$$
Denote by $\mathcal W:=\bigcup_n\mathcal W_n$ the set
of all words, and for $\mathbf a\in\mathcal W_n$,
put $|\mathbf a|:=n$. Denote the empty word by $\emptyset$
and put $\mathcal W^\circ:=\mathcal W\setminus \{\emptyset\}$.
For $\mathbf a=a_1\dots a_n\in\mathcal W$, put
$\overline{\mathbf a}:=\overline{a_n}\dots \overline{a_1}\in\mathcal W$.
If $\mathbf a\in\mathcal W^\circ$,
put $\mathbf a':=a_1\dots a_{n-1}\in\mathcal W$.
Note that $\mathcal W$ forms a tree with root $\emptyset$ and
each $\mathbf a\in\mathcal W^\circ$ having parent $\mathbf a'$.
\item For $\mathbf a=a_1\dots a_n,
\mathbf b=b_1\dots b_m\in\mathcal W$, we write $\mathbf a\to\mathbf b$
if either at least one of $\mathbf a,\mathbf b$
is empty or $a_n\neq \overline{b_1}$. Under this condition
the concatenation $\mathbf a\mathbf b$ is a word.
\item For $\mathbf a,\mathbf b\in \mathcal W$, we write $\mathbf a\prec\mathbf b$
if $\mathbf a$ is a prefix of $\mathbf b$, that is $\mathbf b=\mathbf a\mathbf c$
for some $\mathbf c\in\mathcal W$.
\item For $\mathbf a=a_1\dots a_n,\mathbf b= b_1\dots b_m\in\mathcal W^\circ$,
we write $\mathbf a\rightsquigarrow \mathbf b$ if $a_n=b_1$. Note
that when $\mathbf a\rightsquigarrow\mathbf b$,
the concatenation $\mathbf a'\mathbf b$ is a word of length $n+m-1$.
\item A finite set $Z\subset\mathcal W^\circ$ is called
a \emph{partition} if there exists $N$ such that
for each $\mathbf a\in\mathcal W$ with $|\mathbf a|\geq N$,
there exists unique $\mathbf b\in Z$ such that $\mathbf b\prec\mathbf a$.
\end{itemize}
%
\begin{figure}
\includegraphics{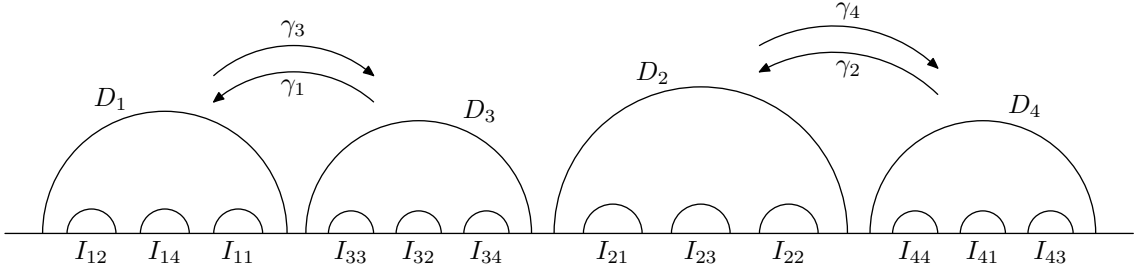}
\caption{A Schottky structure with $r=2$.}
\label{f:schottky}
\end{figure}
For each $\mathbf a=a_1\dots a_n\in\mathcal W$, define the group element
$\gamma_{\mathbf a}\in\Gamma$ by
$$
\gamma_{\mathbf a}:=\gamma_{a_1}\dots\gamma_{a_n}.
$$
Note that each element of $\Gamma$ is equal to $\gamma_{\mathbf a}$
for a unique choice of $\mathbf a$
and $\gamma_{\overline{\mathbf a}}=\gamma_{\mathbf a}^{-1}$,
$\gamma_{\mathbf a\mathbf b}=\gamma_{\mathbf a}\gamma_{\mathbf b}$
when $\mathbf a\to\mathbf b$.

To study the action of $\Gamma$ on $\dot{\mathbb R}$, consider the intervals
$$
I_a:=D_a\cap \dot{\mathbb R}\ \subset\  \mathbb R. 
$$
For each $\mathbf a=a_1\dots a_n\in\mathcal W^\circ$, define the interval $I_{\mathbf a}$ as follows
(see Figure~\ref{f:schottky}):
$$
I_{\mathbf a}:=\gamma_{\mathbf a'}(I_{a_n}).
$$
By~\eqref{e:schottky-mapping}, we have $I_{\mathbf b}\subset I_{\mathbf a}$ when
$\mathbf a\prec \mathbf b$ and $I_{\mathbf a}\cap I_{\mathbf b}=\emptyset$
when $|\mathbf a|=|\mathbf b|$, $\mathbf a\neq \mathbf b$.
The limit set is given by
\begin{equation}
  \label{e:limit-int}
\Lambda_\Gamma:=\bigcap_n\bigsqcup_{\mathbf a\in\mathcal W_n}I_{\mathbf a}.
\end{equation}
A finite set $Z\subset\mathcal W^\circ$ is a partition if and only if
\begin{equation}
  \label{e:partition-intervals}
\Lambda_\Gamma=\bigsqcup_{\mathbf a\in Z} (I_{\mathbf a}\cap \Lambda_\Gamma).
\end{equation}
Denote by $|I|$ the size of an interval $I\subset\mathbb R$.
The following contraction property is proved in~\S\ref{s:der-est-2}:
\begin{equation}
  \label{e:eventually-contracting}
\mathbf a\in\mathcal W^\circ,\
b\in\mathcal A,\
\mathbf a\to b
\quad\Longrightarrow\quad
|I_{\mathbf ab}|\leq (1-C_\Gamma^{-1})|I_{\mathbf a}|.
\end{equation}
Note that~\eqref{e:eventually-contracting} implies the bound
\begin{equation}
  \label{e:eventually-contracting-2}
\mathbf a,\mathbf b\in\mathcal W^\circ,\
\mathbf a\prec\mathbf b
\quad\Longrightarrow\quad
|I_{\mathbf b}|\leq (1-C_\Gamma^{-1})^{|\mathbf b|-|\mathbf a|}|I_{\mathbf a}|
\end{equation}
which gives exponential decay of the sizes of the intervals $I_{\mathbf a}$:
\begin{equation}
  \label{e:eventually-contracting-3}
\mathbf a\in\mathcal W^\circ
\quad\Longrightarrow\quad
|I_{\mathbf a}|\leq C_\Gamma(1-C_\Gamma^{-1})^{|\mathbf a|}.
\end{equation}
We finally describe the collection of words discretizing to a certain resolution.
For $\tau>0$, let $Z(\tau)\subset\mathcal W^\circ$ be defined 
as follows:
\begin{equation}
  \label{e:Z-tau}
Z(\tau)=\{\mathbf a\in\mathcal W^\circ\colon
|I_{\mathbf a}|\leq \tau < |I_{\mathbf a'}|\},
\end{equation}
where we put $|I_{\emptyset}|:=\infty$.
It follows from~\eqref{e:eventually-contracting-3} that $Z(\tau)$ is a partition.
See Figure~\ref{f:partition}.
\begin{figure}
\includegraphics{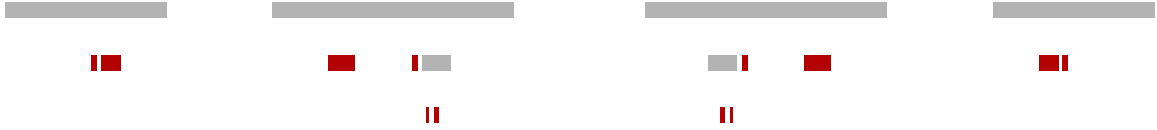}
\caption{A numerically computed example of a partition $Z(\tau)$.
The elements of the partition are in dark red and the preceding intervals
on the tree are in light gray.}
\label{f:partition}
\end{figure}

\subsection{Distortion estimates for M\"obius transformations}
  \label{s:der-est}
  
Let $\mathbf a=a_1\dots a_n$ be a long word. Recall
that $I_{\mathbf a}=\gamma_{\mathbf a'}(I_{a_n})$. In~\S\ref{s:der-est-2} below
 we study
how the derivative $\gamma_{\mathbf a'}$
varies on the interval $I_{a_n}$, in particular
how much it deviates from its average value $|I_{\mathbf a}|/|I_{a_n}|$.
The results of~\S\ref{s:der-est-2} rely on several statements about
general M\"obius transformation which are proved in this section.

Let $\gamma\in \SL(2,\mathbb R)$ and assume that $\gamma(I)=J$
for some intervals $I,J\subset \mathbb R$.
Define the \emph{distortion factor} of~$\gamma$ on~$I$ by
\begin{equation}
  \label{e:distortion-factor}
\alpha(\gamma,I):=\log{\gamma^{-1}(\infty)-x_1\over \gamma^{-1}(\infty)-x_0} \in \mathbb R\quad
\text{where }I=[x_0,x_1].
\end{equation}
If $\gamma^{-1}(\infty)=\infty$, then we put $\alpha(\gamma,I):=0$.
The transformation $\gamma$ can be described in terms of $I$, $J$, and $\alpha(\gamma,I)$ as follows:
\begin{equation}
  \label{e:gammader-1}
\gamma=\gamma_J\,\gamma_{\alpha(\gamma,I)}\,
\gamma_I^{-1},\quad
\gamma_\alpha=\begin{pmatrix} e^{\alpha/2} & 0 \\ e^{\alpha/2}-e^{-\alpha/2} & e^{-\alpha/2} \end{pmatrix}\in \SL(2,\mathbb R).
\end{equation}
Here $\gamma_I,\gamma_J\in\SL(2,\mathbb R)$ are the unique affine transformations
such that $\gamma_I([0,1])=I$, $\gamma_J([0,1])=J$.
To see~\eqref{e:gammader-1}, it suffices to note that
$$
\gamma_J\gamma_{\alpha(\gamma,I)}\gamma_I^{-1}(I)=J,\quad
\gamma_J\gamma_{\alpha(\gamma,I)}\gamma_I^{-1}(\gamma^{-1}(\infty))=\infty.
$$
See Figure~\ref{f:distortion}.
The formula~\eqref{e:gammader-1} implies the following identity:
\begin{equation}
  \label{e:gammader-2}
\gamma'(x)=\gamma'_{\alpha(\gamma,I)}(\gamma_I^{-1}(x))\cdot {|J|\over |I|}.
\end{equation}
Our first lemma states that as long as the distortion factor
is controlled, the derivatives $\gamma'$ at different points of~$I$
do not differ too much from each other and from the average: 
\begin{lemm}
  \label{l:gamma-derivative}
Assume that $\gamma(I)=J$ as above.
Then we have for all $x,y\in I$
\begin{gather}
  \label{e:gamma-derivative-1}
e^{-|\alpha(\gamma,I)|}\cdot{|J|\over|I|}\leq \gamma'(x)\leq e^{|\alpha(\gamma,I)|}\cdot{|J|\over |I|},
\\
  \label{e:gamma-derivative-2}
{\gamma'(x)\over\gamma'(y)}\leq \exp\Big(2e^{|\alpha(\gamma,I)|}\cdot{|x-y|\over |I|}\Big).
\end{gather}
\end{lemm}
\begin{proof}
We estimate for each $\alpha\in\mathbb R$
$$
\gamma'_\alpha(x)={e^\alpha\over ((e^\alpha-1)x+1)^2}\in[e^{-|\alpha|}, e^{|\alpha|}]\quad\text{for }
x\in [0,1]
$$
which together with~\eqref{e:gammader-2} implies~\eqref{e:gamma-derivative-1}.
Next, we have
$$
\big|(\log\gamma_\alpha'(x))'\big|=\Big|{2(1-e^{\alpha})\over (e^\alpha-1)x+1}\Big|\leq 2e^{|\alpha|}\quad\text{for }x\in [0,1]
$$
which gives
$$
{\gamma_\alpha'(x)\over\gamma_\alpha'(y)}\leq \exp\big(
2e^{|\alpha|}\cdot |x-y|
\big)\quad\text{for }x,y\in [0,1].
$$
Combined with~\eqref{e:gammader-2}, this implies~\eqref{e:gamma-derivative-2}.
\end{proof}
%
\begin{figure}
\includegraphics{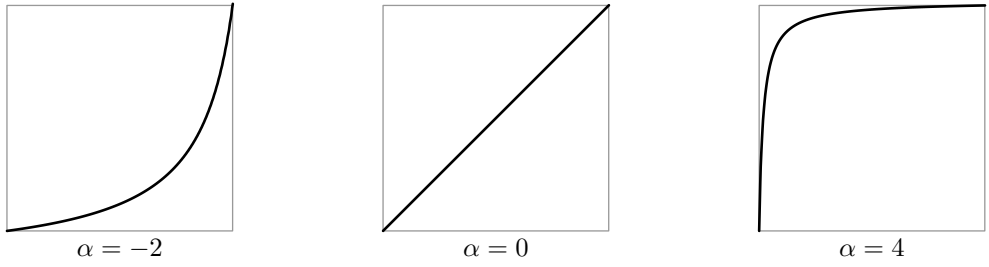}
\caption{Graphs of the transformation $\gamma_\alpha$ for several different
values of $\alpha$, with the square being $[0,1]^2$.}
\label{f:distortion}
\end{figure}
As a corollary of~\eqref{e:gamma-derivative-1} and the change of variable formula,
we immediately obtain
\begin{lemm}
  \label{l:gamma-integral}
Assume that $\gamma(I)=J$ as above and let $I'\subset I$ be a Borel subset.
Then, denoting by $|\bullet|$ the Lebesgue measure on the line, we have
\begin{equation}
  \label{e:gamma-integral}
e^{-|\alpha(\gamma,I)|}\cdot{|I'|\cdot |J|\over|I|}\leq |\gamma(I')|\leq e^{|\alpha(\gamma,I)|}\cdot{|I'|\cdot|J|\over |I|}.
\end{equation}
\end{lemm}
The next lemma shows that transformations with different
distortion factors have significantly different derivatives.
It is an essential component of the proof of Theorem~\ref{t:fourier} which takes advantage
of the nonlinearity of M\"obius transformations.
\begin{lemm}
  \label{l:distortion-discrepancy}
Assume that $\gamma_1,\gamma_2\in\SL(2,\mathbb R)$ and
$I,J_1,J_2\subset \mathbb R$ are intervals such that
$\gamma_j(I)=J_j$. Let $L\subset \mathbb R$ be an interval.
Then the set of points $x$ satisfying
\begin{equation}
  \label{e:distortion-discrepancy}
x\in I,\quad
\log{\gamma_1'(x)\over\gamma_2'(x)}\in L
\end{equation}
is contained in an interval of size
$$
{e^{|\alpha(\gamma_1,I)|+|\alpha(\gamma_2,I)|}\cdot |I|\cdot |L|\over |\alpha(\gamma_1,I)-\alpha(\gamma_2,I)|}.
$$  
\end{lemm}
\begin{proof}
Denote $\alpha_j=\alpha(\gamma_j,I)$. For each $x\in I$ we have by~\eqref{e:gammader-2}
$$
\log{\gamma_1'(x)\over\gamma_2'(x)}=
\log{\gamma_{\alpha_1}'(y)\over\gamma_{\alpha_2}'(y)}+
\log{|J_1|\over |J_2|},\quad
y:=\gamma_I^{-1}(x).
$$
Therefore, \eqref{e:distortion-discrepancy}
corresponds to the set of all $y$ such that
\begin{equation}
  \label{e:dd-int}
y\in [0,1],\quad
\log{\gamma_{\alpha_1}'(y)\over\gamma_{\alpha_2}'(y)}\in \widetilde L
\end{equation}
where $\widetilde L$ is some interval with $|\widetilde L|=|L|$.
We compute
$$
\partial_y \log{\gamma_{\alpha_1}'(y)\over\gamma_{\alpha_2}'(y)}=
{2(1-e^{\alpha_1})\over (e^{\alpha_1}-1)y+1}
-{2(1-e^{\alpha_2})\over (e^{\alpha_2}-1)y+1}
={2(e^{\alpha_2}-e^{\alpha_1})\over ((e^{\alpha_1}-1)y+1)((e^{\alpha_2}-1)y+1)}.
$$
We then have for all $y\in [0,1]$
$$
\bigg|\partial_y \log{\gamma_{\alpha_1}'(y)\over\gamma_{\alpha_2}'(y)}\bigg|\geq
2e^{-|\alpha_1|-|\alpha_2|}\cdot |\alpha_1-\alpha_2|.
$$
It follows that the set of $y$ satisfying~\eqref{e:dd-int}
is an interval of size no more than
$$
{e^{|\alpha_1|+|\alpha_2|}\cdot |L|\over |\alpha_1-\alpha_2|}
$$
which finishes the proof.
\end{proof}

\subsection{Distortion estimates for Schottky groups}
  \label{s:der-est-2}

We now return to the setting of Schottky groups introduced in~\S\ref{s:schottky}.
We start by estimating the distortion factors
of transformations in $\Gamma$:
\begin{lemm}
  \label{l:alpha-bound}
We have
\begin{equation}
  \label{e:alpha-bound}
|\alpha(\gamma_{\mathbf a},I_b)|\leq C_\Gamma\quad\text{for all}\quad
\mathbf a\in\mathcal W,\ b\in\mathcal A,\
\mathbf a\to b.
\end{equation}
\end{lemm}
\begin{proof}
We may assume that $\mathbf a\in\mathcal W^\circ$.
Let
$\mathbf a=a_1\dots a_n$. By~\eqref{e:schottky-mapping}, $\gamma_{\mathbf a}^{-1}(\infty)\in I_{\overline{a_n}}$.
Moreover, $\overline{a_n}\neq b$ since $\mathbf a\to b$. It remains to recall
the definition~\eqref{e:distortion-factor} and put
$$
C_\Gamma:=
2\max\big\{
\big|\log |x-y|\big|\colon
x\in I_a,\ 
y\in I_b,\
a,b\in\mathcal A,\
a\neq b
\big\}.
\qedhere
$$ 
\end{proof}
Lemma~\ref{l:alpha-bound} together with~\eqref{e:gamma-derivative-1},
\eqref{e:gamma-derivative-2}, and~\eqref{e:gamma-integral} immediately gives
\begin{lemm}
  \label{l:tree-derivative}
For all $\mathbf a=a_1\dots a_n\in\mathcal W^\circ$ and
$x,y\in I_{a_n}$, we have
\begin{gather}
  \label{e:tree-derivative-1}
C_\Gamma^{-1}|I_{\mathbf a}|\leq \gamma_{\mathbf a'}'(x)\leq C_\Gamma|I_{\mathbf a}|,
\\
  \label{e:tree-derivative-2}
{\gamma_{\mathbf a'}'(x)\over \gamma_{\mathbf a'}'(y)}\leq \exp\big(C_\Gamma|x-y|\big).
\end{gather}
Moreover, if $I'\subset I_{a_n}$ is a Borel set, then
\begin{equation}
  \label{e:tree-integral}
C_\Gamma^{-1}|I_{\mathbf a}|\cdot |I'|\leq |\gamma_{\mathbf a'}(I')|\leq C_\Gamma|I_{\mathbf a}|\cdot |I'|.
\end{equation}
\end{lemm}
Armed with Lemma~\ref{l:tree-derivative}, we give
\begin{proof}[Proof of~\eqref{e:eventually-contracting}]
We write $\mathbf a=a_1\dots a_n$. With $|\bullet|$ denoting the Lebesgue
measure on the line, we compute
$$
|I_{\mathbf ab}|=|\gamma_{\mathbf a'}(\gamma_{a_n}(I_b))|
=|\gamma_{\mathbf a'}(I_{a_n})|-|\gamma_{\mathbf a'}(I_{a_n}\setminus \gamma_{a_n}(I_b))|.
$$
Recall that $\gamma_{\mathbf a'}(I_{a_n})=I_{\mathbf a}$. Using~\eqref{e:tree-integral},
we obtain the lower bound
$$
|\gamma_{\mathbf a'}(I_{a_n}\setminus \gamma_{a_n}(I_b))|
\geq C_\Gamma^{-1}|I_{\mathbf a}|\cdot |I_{a_n}\setminus \gamma_{a_n}(I_b)|
\geq C_\Gamma^{-1}|I_{\mathbf a}|
$$
finishing the proof.
\end{proof}
We next show several estimates on the sizes and positions of the intervals $I_{\mathbf a}$:
\begin{lemm}[Parent-child ratio]
  \label{l:tree-regularity}
We have
\begin{equation}
  \label{e:tree-regularity}
C_\Gamma^{-1}|I_{\mathbf a}|\leq |I_{\mathbf ab}|\leq |I_{\mathbf a}|\quad\text{for all }
\mathbf a\in\mathcal W^\circ,\ b\in\mathcal A,\
\mathbf a\to b.
\end{equation}
\end{lemm}
\begin{proof}
Denote $\mathbf a=a_1\dots a_n$ and note that
$I_{\mathbf ab}=\gamma_{\mathbf a'}(I')$
where $I':=\gamma_{a_n}(I_b)\subset I_{a_n}$.
Then~\eqref{e:tree-regularity} follows from~\eqref{e:tree-integral}.
\end{proof}
%
\begin{lemm}[Concatenation]
  \label{l:concatenate}
We have
\begin{equation}
  \label{e:concatenate}
C_\Gamma^{-1}|I_{\mathbf a}|\cdot|I_{\mathbf b}|\leq |I_{\mathbf a'\mathbf b}|\leq C_\Gamma|I_{\mathbf a}|\cdot|I_{\mathbf b}|\quad\text{
for all }\mathbf a,\mathbf b\in\mathcal W^\circ,\
\mathbf a\rightsquigarrow\mathbf b.
\end{equation}  
\end{lemm}
\begin{proof}
This follows from~\eqref{e:tree-integral} similarly to Lemma~\ref{l:tree-regularity},
using that $I_{\mathbf a'\mathbf b}=\gamma_{\mathbf a'}(I_{\mathbf b})$.
\end{proof}
%
\begin{lemm}[Reversal]
  \label{l:reversible}
We have
\begin{equation}
  \label{e:reversible}
C_\Gamma^{-1}|I_{\mathbf a}| \leq |I_{\overline{\mathbf a}}|\leq C_\Gamma|I_{\mathbf a}|
\quad\text{for all }\mathbf a\in \mathcal W^\circ.
\end{equation}
\end{lemm}
\begin{proof}
Without loss of generality, we may assume that $|\mathbf a|\geq 3$.
We write $\mathbf a=a_1\dots a_n$ and denote $\mathbf b := a_2\dots a_{n-1}$,
so that $\mathbf a=a_1\mathbf b a_n$.
Since 
$I_{\mathbf a}=\gamma_{a_1}(I_{\mathbf b a_n})$ and
$I_{\overline{\mathbf a}}=\gamma_{\overline{a_n}}(I_{\overline{\mathbf b}\overline{a_1}})$,
it suffices to show that
\begin{equation}
  \label{e:reversible-int}
C_\Gamma^{-1}|I_{\mathbf b a_n}| \leq 
|I_{\overline{\mathbf b}\overline{a_1}}|
\leq C_\Gamma|I_{\mathbf b a_n}|.
\end{equation}
Denote
$$
I_{a_n}=[x_1,x_2],\quad
I_{\overline{\mathbf b}\overline{a_1}}=[x_3,x_4],\quad
I_{\mathbf b a_n}=[y_1,y_2],\quad
I_{\overline{a_1}}=[y_3,y_4]
$$
and remark that $\gamma_{\mathbf b}(x_j)=y_j$ and thus
we have equality of cross ratios
\begin{equation}
  \label{e:reversible-int2}
{(x_2-x_1)(x_4-x_3)\over (x_4-x_1)(x_3-x_2)}
={(y_2-y_1)(y_4-y_3)\over (y_4-y_1)(y_3-y_2)}.
\end{equation}
Now, $x_3,x_4\in I_{\overline{a_{n-1}}}$ and $\overline{a_{n-1}}\neq a_n$.
Therefore,
$$
|x_2-x_1|,|x_4-x_1|,|x_3-x_2|\in [C_\Gamma^{-1},C_\Gamma].
$$
Since $y_1,y_2\in I_{a_2}$ we similarly bound $|y_4-y_3|,|y_4-y_1|,|y_3-y_2|$.
Then~\eqref{e:reversible-int} follows from~\eqref{e:reversible-int2}
and the fact that
$|I_{\mathbf b a_n}|=y_2-y_1$,
$|I_{\overline{\mathbf b}\overline{a_1}}|=x_4-x_3$.
\end{proof}
%
\begin{lemm}[Separation]
  \label{l:separation}
Assume that $\mathbf a,\mathbf b\in\mathcal W^\circ$
and $\mathbf a\not\prec\mathbf b$, $\mathbf b\not\prec\mathbf a$. Then
\begin{equation}
  \label{e:separation}
|x-y|
\geq C_\Gamma^{-1}\max\big(|I_{\mathbf a}|,|I_{\mathbf b}|\big)\quad\text{for all }
x\in I_{\mathbf a},\
y\in I_{\mathbf b}.
\end{equation}
\end{lemm}
\begin{proof}
Since $\mathbf a\not\prec\mathbf b$, $\mathbf b\not\prec\mathbf a$, there
exist
$$
\mathbf c\in \mathcal W,\
d,e\in\mathcal A\quad\text{such that}\quad
\mathbf c\to d,\
\mathbf c\to e,\
\mathbf cd\prec\mathbf a,\
\mathbf ce\prec\mathbf b,\
d\neq e.
$$
Without loss of generality we may assume that $\mathbf c\in\mathcal W^\circ$
and write $\mathbf c=c_1\dots c_n$.
Then
$$
I_{\mathbf a}\subset I_{\mathbf cd}=\gamma_{\mathbf c'}(\gamma_{c_n}(I_d)),
\quad
I_{\mathbf b}\subset I_{\mathbf ce}=\gamma_{\mathbf c'}(\gamma_{c_n}(I_e)).
$$
Since the distance between $\gamma_{c_n}(I_d)$ and $\gamma_{c_n}(I_e)$ is bounded below
by $C_\Gamma^{-1}$
and both these intervals are contained in $I_{c_n}$,
we get by~\eqref{e:tree-derivative-1}
$$
|x-y|\geq C_\Gamma^{-1}|I_{\mathbf c}|
\geq C_\Gamma^{-1}\max\big(|I_{\mathbf a}|,|I_{\mathbf b}|\big)\quad\text{for all }
x\in I_{\mathbf a},\
y\in I_{\mathbf b}
$$
finishing the proof.
\end{proof}
We finally obtain estimates on the elements of
the partition $Z(\tau)$ defined in~\eqref{e:Z-tau}:
\begin{lemm}
For all $\tau\in (0,1]$ and $\mathbf a=a_1\dots a_n\in Z(\tau)$, we have
\begin{gather}
  \label{e:Z-tau-prop-1}
C_\Gamma^{-1}\tau \leq |I_{\mathbf a}|\leq \tau,\\
  \label{e:Z-tau-prop-1.5}
C_\Gamma^{-1}\tau \leq |I_{\overline{\mathbf a}}|\leq C_\Gamma\,\tau,\\
  \label{e:Z-tau-prop-2}
C_{\Gamma}^{-1}\tau\leq \gamma_{\mathbf a'}'\leq C_\Gamma\, \tau\quad\text{on }I_{a_n}.
\end{gather}
\end{lemm}
\begin{proof}
Let $\mathbf a\in Z(\tau)$. Without loss of generality we may assume that $|\mathbf a|\geq 2$.
We have $|I_{\mathbf a}|\leq \tau < |I_{\mathbf a'}|$
and by Lemma~\ref{l:tree-regularity}, $|I_{\mathbf a}|\geq C_\Gamma^{-1}|I_{\mathbf a'}|$.
This gives~\eqref{e:Z-tau-prop-1}.
Now~\eqref{e:Z-tau-prop-1.5} follows from~\eqref{e:reversible}, and
\eqref{e:Z-tau-prop-2} follows from~\eqref{e:tree-derivative-1}.
\end{proof}

\subsection{Patterson--Sullivan measure}
  \label{s:ps}

The Patterson--Sullivan measure $\mu$ is equivariant under the group $\Gamma$:
for any bounded Borel function $f$ on $\mathbb R$,
\begin{equation}
  \label{e:psin-1}
\int_{\Lambda_\Gamma} f(x)\,d\mu(x)
=\int_{\Lambda_\Gamma} f(\gamma(x))|\gamma'(x)|_{\mathbb B}^\delta \,d\mu(x)\quad\text{for all }
\gamma\in\Gamma
\end{equation}
where $|\gamma'|_{\mathbb B}$ is the derivative of $\gamma$ as a map
of the ball model of the hyperbolic space:
$$
|\gamma'(x)|_{\mathbb B}={1+x^2\over 1+\gamma(x)^2}\,\gamma'(x),\quad
x\in\dot{\mathbb R}.
$$
See for instance~\cite[Lemma~14.2]{BorthwickBook}. Next, \eqref{e:psin-1}
implies
\begin{equation}
  \label{e:psin-2}
\int_{I_{\mathbf a b}}f(x)\,d\mu(x)
=\int_{I_b} f(\gamma_{\mathbf a}(x))w_{\mathbf a}(x)\,d\mu(x)\quad\text{for all }
\mathbf a\in\mathcal W,\
b\in\mathcal A,\
\mathbf a\to b
\end{equation}
where the weight $w_{\mathbf a}$ is defined by
\begin{equation}
  \label{e:w-a}
w_{\mathbf a}(x):=|\gamma_{\mathbf a}'(x)|_{\mathbb B}^\delta.
\end{equation}
The Patterson--Sullivan measure of an interval $I_{\mathbf a}$ is estimated
by the following
\begin{lemm}
  \label{l:ad-int}
We have
\begin{equation}
  \label{e:ad-int}
C_\Gamma^{-1}|I_{\mathbf a}|^\delta
\leq \mu(I_{\mathbf a})
\leq C_\Gamma |I_{\mathbf a}|^\delta
\quad\text{for all }
\mathbf a\in\mathcal W^\circ.
\end{equation}
\end{lemm}
\begin{proof}
The formula~\eqref{e:psin-2} implies that for all $a,b\in\mathcal A$,
$a\neq \overline b$, we have
$$
\mu(I_a)\geq \mu(I_{ab})=\int_{I_b}w_{a}(x)\,d\mu(x)\geq C_\Gamma^{-1}\mu(I_b).
$$
Since $\mu$ is a probability measure, this implies that
$$
C_\Gamma^{-1}\leq \mu(I_a)\leq 1\quad\text{for all }a\in\mathcal A.
$$
Denote $\mathbf a=a_1\dots a_n$.
From~\eqref{e:psin-2} we have
$$
\mu(I_{\mathbf a})=\int_{I_{a_n}}w_{\mathbf a'}(x)\,d\mu(x).
$$
By~\eqref{e:tree-derivative-1} we have
$$
C_\Gamma^{-1}|I_{\mathbf a}|^\delta
\leq  w_{\mathbf a'}
\leq C_\Gamma |I_{\mathbf a}|^\delta\quad\text{on }I_{a_n}
$$
and~\eqref{e:ad-int} follows.
\end{proof}
Using Lemma~\ref{l:ad-int}, we give a self-contained
proof of Ahlfors--David regularity of $\mu$ (see~\cite[Lemma~14.13]{BorthwickBook}
for another proof):
\begin{lemm}
  \label{l:ad-regularity}
Let $I\subset\mathbb R$ be an interval.
Then
\begin{equation}
  \label{e:adreg-1}
\mu(I)\leq C_\Gamma |I|^\delta.
\end{equation}
If additionally $|I|\leq 1$ and $I$ is centered at a point in $\Lambda_\Gamma$, then
\begin{equation}
  \label{e:adreg-2}
\mu(I)\geq C_\Gamma^{-1} |I|^\delta.
\end{equation}
\end{lemm}
\begin{proof}
We first show the upper bound~\eqref{e:adreg-1}.
Since $\mu$ is supported on $\Lambda_\Gamma$,
replacing $I$ with the intersections $I\cap I_a$
we may assume that $I\subset I_a$ for some $a\in\mathcal A$.
Shrinking $I$ without changing $\mu(I)$, we may also assume that its endpoints $x_0,x_1$ lie in $\Lambda_\Gamma$.
If $I=\{x_0\}$ consists of one point, then by~\eqref{e:limit-int} we can find arbitrarily
long words $\mathbf a$ such that $x_0\in I_{\mathbf a}$;
by~\eqref{e:eventually-contracting-3} and~\eqref{e:ad-int},
we have $\mu(I)=0$.

Assume now that $x_0<x_1$. By~\eqref{e:eventually-contracting-3}
there exists the longest word $\mathbf a=a_1\dots a_n\in\mathcal W^\circ$ such that
$I\subset I_{\mathbf a}$. Then $x_0\in I_{\mathbf a b}$,
$x_1\in I_{\mathbf a c}$ for two different $b,c\in\mathcal A$
such that $\mathbf a\to b$, $\mathbf a\to c$.
By Lemma~\ref{l:separation}, the distance between $I_{\mathbf ab}$
and $I_{\mathbf ac}$ is
bounded below by $C_\Gamma^{-1}|I_{\mathbf a}|$, therefore
$|I|\geq C_\Gamma^{-1}|I_{\mathbf a}|$. Now~\eqref{e:adreg-1} follows from~\eqref{e:ad-int}:
$$
\mu(I)\leq \mu(I_{\mathbf a})\leq C_\Gamma |I_{\mathbf a}|^\delta
\leq C_\Gamma |I|^\delta.
$$
We next show the lower bound~\eqref{e:adreg-2}
where $I$ is an interval of size $0<|I|\leq 1$ centered at
some $x\in \Lambda_\Gamma$. Using~\eqref{e:eventually-contracting-3}, take the shortest word $\mathbf a\in\mathcal W^\circ$
such that $x\in I_{\mathbf a}\subset I$. If $|\mathbf a|=1$, then
by~\eqref{e:ad-int} $\mu(I)\geq \mu(I_{\mathbf a})\geq C_\Gamma^{-1}$.
Assume now that $|\mathbf a|\geq 2$.

Since $x\in I_{\mathbf a'}$ and $I_{\mathbf a'}\not\subset I$,
we have $|I_{\mathbf a'}|\geq {1\over 2}|I|$ and thus
by~\eqref{e:tree-regularity}
$|I_{\mathbf a}|\geq C_\Gamma^{-1}|I|$. Now~\eqref{e:adreg-2} follows
from~\eqref{e:ad-int}:
$$
\mu(I)\geq \mu(I_{\mathbf a})\geq C_\Gamma^{-1}|I_{\mathbf a}|^\delta
\geq C_\Gamma^{-1}|I|^\delta.\qedhere
$$
\end{proof}
As another corollary of Lemma~\ref{l:ad-int}, we estimate the number of elements
in the partition $Z(\tau)$ defined in~\eqref{e:Z-tau}:
\begin{lemm}
  \label{l:partition-count}
For $\tau\in (0,1]$ we have
\begin{equation}
  \label{e:partition-count}
C_\Gamma^{-1}\tau^{-\delta}\leq \#(Z(\tau))\leq C_\Gamma\,\tau^{-\delta}.
\end{equation}
\end{lemm}
\begin{proof}
Since $Z(\tau)$ is a partition, we have by~\eqref{e:partition-intervals}
$$
1=\mu(\Lambda_\Gamma)=\sum_{\mathbf a\in Z(\tau)}\mu(I_{\mathbf a}).
$$
By~\eqref{e:Z-tau-prop-1} and~\eqref{e:ad-int}, we have for all $\mathbf a\in Z(\tau)$
\begin{equation}
  \label{e:pc-int}
C_\Gamma^{-1}\tau^\delta\leq \mu(I_{\mathbf a})\leq C_\Gamma\,\tau^\delta
\end{equation}
which implies~\eqref{e:partition-count}.
%
%
\end{proof}
The following is an analogue of the upper bound of Lemma~\ref{l:ad-int} where
instead of the measure $\mu(I_{\mathbf b})$ we estimate
the number of intervals of length at least $\tau$
in the subtree with root $I_{\mathbf b}$:
\begin{lemm}
  \label{l:partition-count-1.5}
Assume that $\tau\in (0,1]$, $\mathbf b\in\mathcal W^\circ$.
Then
\begin{equation}
  \label{e:pc-int-2}
\#\{\mathbf a\in\mathcal W^\circ\colon
\mathbf b\prec\mathbf a,\
|I_{\mathbf a}|\geq\tau\}\leq C_\Gamma\, \tau^{-\delta}|I_{\mathbf b}|^\delta.
\end{equation}
\end{lemm}
\begin{proof}
We may assume that $|I_{\mathbf b}|\geq\tau$ since
otherwise the left-hand side of~\eqref{e:pc-int-2} equals~0.
By~\eqref{e:eventually-contracting-3}, the following sets are finite:
$$
A:=\{\mathbf a\in\mathcal W^\circ\colon \mathbf b\prec\mathbf a,\
|I_{\mathbf a}|\geq\tau\},\quad
B:=\{\mathbf a\in\mathcal W^\circ\colon
\mathbf b\prec\mathbf a,\
|I_{\mathbf a}|<\tau\leq |I_{\mathbf a'}|\}.
$$
Then $\{I_{\mathbf a}\}_{\mathbf a\in B}$ is a disjoint collection
of subintervals of $I_{\mathbf b}$. Therefore by~\eqref{e:ad-int}
$$
\sum_{\mathbf a\in B} \mu(I_{\mathbf a})\leq \mu(I_{\mathbf b})
\leq C_\Gamma |I_{\mathbf b}|^\delta.
$$
On the other hand, by~\eqref{e:tree-regularity} and~\eqref{e:ad-int}
$$
\mu(I_{\mathbf a})\geq C_\Gamma^{-1}|I_{\mathbf a}|^\delta
\geq C_\Gamma^{-1}\tau^\delta\quad\text{for all }\mathbf a\in B.
$$
Therefore, the number of elements in $B$ is bounded as follows:
\begin{equation}
  \label{e:pc-int-3}
\#(B)\leq C_\Gamma\, \tau^{-\delta}|I_{\mathbf b}|^\delta.
\end{equation}
Next, $A\sqcup B$ forms a tree with root $\mathbf b$,
where the parent of $\mathbf a$ is given by $\mathbf a'$.
Moreover, $B$ is the set of leaves of this tree
and each element of $A$ has exactly $2r-1$ children,
where $2r\geq 4$ is the number of intervals in the Schottky structure.
The number of edges of the tree is equal to
both $\#(A)+\#(B)-1$ and $(2r-1)\cdot \#(A)$, which implies
$$
\#(A)={\#(B)-1\over 2r-2}\leq \#(B).
$$
Combining this with~\eqref{e:pc-int-3}, we obtain~\eqref{e:pc-int-2}.
\end{proof}
Arguing similarly to the proof of~\eqref{e:adreg-1},
we obtain from Lemma~\ref{l:partition-count-1.5} the following
\begin{lemm}
  \label{l:partition-count-2}
For all intervals $J$ and all $C_0\geq 2$ we have
\begin{equation}
  \label{e:partition-count-2}
\#\big\{\mathbf a\in \mathcal W^\circ \colon \tau\leq |I_{\mathbf a}|\leq C_0\tau,\
I_{\mathbf a}\cap J\neq\emptyset\big\}
\leq C_\Gamma\, \tau^{-\delta} |J|^\delta+C_\Gamma\log C_0.
\end{equation} 
\end{lemm}
\begin{proof}
Without loss
of generality we may assume that $J$ is contained in $I_a$ for some $a\in\mathcal A$.
Consider the finite set
$$
X:=\{\mathbf a\in\mathcal W^\circ\colon |I_{\mathbf a}|\geq \tau,\
I_{\mathbf a}\cap J\neq\emptyset\}.
$$
Then $X$ forms a tree with root $a$ in the sense that $\mathbf a\in X\setminus\{a\}$ implies
$\mathbf a'\in X$.

Take the longest word $\mathbf b\in X$ with the following
property: for each $\mathbf a\in X$, we have $\mathbf a\prec\mathbf b$ or $\mathbf b\prec\mathbf a$.
Then $\mathbf b$ cannot have exactly one child in $X$, leaving the following two options:
\begin{enumerate}
\item $\mathbf b$ has no children in $X$. Then all $\mathbf a\in X$ satisfy
$\mathbf a\prec\mathbf b$.
By~\eqref{e:eventually-contracting-2}, we estimate the
number of elements $\mathbf a\in X$ such that $|I_{\mathbf a}|\leq C_0\tau$
by $C_\Gamma\log C_0$.
\item There exist $c,d\in\mathcal A$, $c\neq d$, $\mathbf b\to c$,
$\mathbf b\to d$, such that $\mathbf bc,\mathbf bd\in X$.
By Lemma~\ref{l:separation} the distance between $I_{\mathbf bc}$ and $I_{\mathbf bd}$
is bounded below by $C_\Gamma^{-1}|I_{\mathbf b}|$, and both these intervals intersect
$J$, therefore
$$
|I_{\mathbf b}|\leq C_\Gamma |J|.
$$
By~\eqref{e:pc-int-2}, the number of elements $\mathbf a\in X$
such that $\mathbf b\prec\mathbf a$ is bounded above by $C_\Gamma\,\tau^{-\delta}|J|^\delta$.
All other elements $\mathbf a\in X$ have to satisfy $\mathbf a\prec\mathbf b$, 
and arguing similarly to the previous case we see that the number of these
with $|I_{\mathbf a}|\leq C_0\tau$ is bounded above by $C_\Gamma\log C_0$.\qedhere
\end{enumerate}
\end{proof}
We finally use Lemma~\ref{l:distortion-discrepancy}
to obtain the following statement, which gives
the positive box dimension estimate required in~\S\ref{s:fourier-2}.
This is the only statement which uses
both Lemma~\ref{l:reversible} (via~\eqref{e:Z-tau-prop-1.5})
and the full power of Lemma~\ref{l:partition-count-2}.
Recall the notation $\mathbf a\rightsquigarrow \mathbf b$
from~\S\ref{s:schottky}. We introduce the following additional
piece of notation:
\begin{equation}
  \label{e:quads}
\mathbf a\rightsquigarrow {\textstyle{\mathbf b\atop \mathbf c}}\rightsquigarrow \mathbf d\quad\text{if and only if}\quad
\mathbf a\rightsquigarrow\mathbf b\rightsquigarrow\mathbf d\text{ and }
\mathbf a\rightsquigarrow\mathbf c\rightsquigarrow\mathbf d.
\end{equation}
\begin{lemm}
  \label{l:triple-count}
Fix $\mathbf a\in Z(\tau)$
and for each $\mathbf d\in\mathcal W^\circ$ let $x_{\mathbf d}$
be the center of $I_{\mathbf d}$.
Then we have for $0<\tau\leq\sigma\leq 1$
\begin{equation}
  \label{e:triple-count}
\begin{gathered}
\#\big\{
(\mathbf b,\mathbf c,\mathbf d)\in Z(\tau)^3
\colon
\mathbf a\rightsquigarrow {\textstyle{\mathbf b\atop \mathbf c}}\rightsquigarrow \mathbf d,\
|\gamma'_{\mathbf a'\mathbf b'}(x_{\mathbf d})-\gamma'_{\mathbf a'\mathbf c'}(x_{\mathbf d})|
\leq \tau^2\sigma
\big
\}
\\
\leq C_\Gamma\, \tau^{-3\delta}\sigma^{\delta/2}.
\end{gathered}\end{equation}  
\end{lemm}
\begin{proof}
Without loss of generality, we may assume that
$\tau$ is small enough so that $|\mathbf c|\geq 2$ for all $\mathbf c\in Z(\tau)$.
For each $\mathbf b\in Z(\tau)$ such that
$\mathbf a\rightsquigarrow\mathbf b$, we have
\begin{equation}
  \label{e:different-distort}
\#\big\{\mathbf c\in Z(\tau)
\colon
\mathbf a\rightsquigarrow \mathbf c,\
|\gamma_{\mathbf a'\mathbf b'}^{-1}(\infty)
-\gamma_{\mathbf a'\mathbf c'}^{-1}(\infty)|
\leq \sqrt\sigma
\big\}
\leq C_\Gamma\, \tau^{-\delta}\sigma^{\delta/2}.
\end{equation}
Indeed, denoting $\mathbf e:=\overline{\mathbf c'}$, we have
$\gamma_{\mathbf a'\mathbf c'}^{-1}(\infty)=\gamma_{\mathbf e\overline{\mathbf a'}}(\infty)\in I_{\mathbf e}$.
Also, $C_\Gamma^{-1}\tau\leq |I_{\mathbf e}|\leq C_\Gamma\tau$ by~\eqref{e:Z-tau-prop-1.5} and~\eqref{e:concatenate}.
Therefore, the left-hand side of~\eqref{e:different-distort} is bounded by
$$
2r\cdot \#\big\{\mathbf e\in \mathcal W^\circ\colon
C_\Gamma^{-1}\tau\leq |I_{\mathbf e}|\leq C_\Gamma\tau,\
I_{\mathbf e}\cap J\neq\emptyset
\big\},\quad
J:=\gamma_{\mathbf a'\mathbf b'}^{-1}(\infty)+[-\sqrt\sigma,\sqrt\sigma].
$$
Now \eqref{e:different-distort} follows from~\eqref{e:partition-count-2}.

By~\eqref{e:different-distort} and~\eqref{e:partition-count}, the triples
$(\mathbf b,\mathbf c,\mathbf d)$ with 
$|\gamma_{\mathbf a'\mathbf b'}^{-1}(\infty)
-\gamma_{\mathbf a'\mathbf c'}^{-1}(\infty)|
\leq \sqrt\sigma$
contribute at most $C_\Gamma\tau^{-3\delta}\sigma^{\delta/2}$ to the left-hand side of~\eqref{e:triple-count}.
Therefore, it remains to show that for each $\mathbf b,\mathbf c\in Z(\tau)$ such that
$\mathbf a\rightsquigarrow\mathbf b$, $\mathbf a\rightsquigarrow\mathbf c$ and
\begin{equation}
  \label{e:different-distort-2}
|\gamma_{\mathbf a'\mathbf b'}^{-1}(\infty)
-\gamma_{\mathbf a'\mathbf c'}^{-1}(\infty)|
\geq \sqrt\sigma,
\end{equation}
we have
\begin{equation}
  \label{e:tripc-int}
\#\big\{
\mathbf d\in Z(\tau)\colon
\mathbf b\rightsquigarrow\mathbf d,\
\mathbf c\rightsquigarrow\mathbf d,\
|\gamma'_{\mathbf a'\mathbf b'}(x_{\mathbf d})-\gamma'_{\mathbf a'\mathbf c'}(x_{\mathbf d})|\leq\tau^2\sigma
\big\}\leq C_\Gamma\, \tau^{-\delta} \sigma^{\delta/2}.
\end{equation}
Denote by $b_n$ the last letter of $\mathbf b$; we may assume it is also the last letter of $\mathbf c$,
since otherwise the left-hand side of~\eqref{e:tripc-int} is zero.

By~\eqref{e:Z-tau-prop-1} and~\eqref{e:concatenate} we have
$C_\Gamma^{-1}\tau^2\leq |I_{\mathbf a'\mathbf b}|\leq C_\Gamma\tau^2$
and
$C_\Gamma^{-1}\tau^2\leq |I_{\mathbf a'\mathbf c}|\leq C_\Gamma\tau^2$.
By~\eqref{e:tree-derivative-1} this gives
$C_\Gamma^{-1}\tau^2\leq \gamma'_{\mathbf a'\mathbf b'}\leq C_\Gamma\tau^2$ and
$C_\Gamma^{-1}\tau^2\leq\gamma'_{\mathbf a'\mathbf c'}\leq C_\Gamma\tau^2$
on $I_{b_n}$. Thus it suffices to show that for any given constant $C_0$
depending only on the Schottky data,
\begin{equation}
  \label{e:tripc-int2}
\#\Big\{
\mathbf d\in Z(\tau)\colon
\mathbf b\rightsquigarrow\mathbf d,\
\mathbf c\rightsquigarrow\mathbf d,\
\Big|\log{\gamma'_{\mathbf a'\mathbf b'}(x_{\mathbf d})\over \gamma'_{\mathbf a'\mathbf c'}(x_{\mathbf d})}\Big|\leq C_0\sigma
\Big\}\leq C_\Gamma\, \tau^{-\delta} \sigma^{\delta/2}.
\end{equation}
By~\eqref{l:alpha-bound},
\eqref{e:distortion-factor},
and~\eqref{e:different-distort-2},
we have
$$
|\alpha(\gamma_{\mathbf a'\mathbf b'},I_{b_n})|,
|\alpha(\gamma_{\mathbf a'\mathbf c'},I_{b_n})|
\leq C_\Gamma,\quad
|\alpha(\gamma_{\mathbf a'\mathbf b'},I_{b_n})-\alpha(\gamma_{\mathbf a'\mathbf c'},I_{b_n})|\geq C_\Gamma^{-1}\sqrt{\sigma}.
$$
By Lemma~\ref{l:distortion-discrepancy},
there exists an interval $\widetilde J$
of size $C_\Gamma\sqrt{\sigma}$ depending on $\mathbf a,\mathbf b,\mathbf c$ such that
 for each $\mathbf d$ on the left-hand side of~\eqref{e:tripc-int2},
the point $x_{\mathbf d}$ lies in $\widetilde J$
and thus $I_{\mathbf d}\cap\widetilde J\neq\emptyset$.
 Then by~\eqref{e:partition-count-2} and~\eqref{e:Z-tau-prop-1} we obtain~\eqref{e:tripc-int2},
finishing the proof.
\end{proof}

\subsection{Transfer operators}
  \label{s:transfer}

For a partition $Z\subset\mathcal W^\circ$, define
the operator
$$
\mathcal L_Z:\Bor(\mathcal I)\to \Bor(\mathcal I),\quad
\mathcal I:=\bigsqcup_{b\in\mathcal A}I_b,
$$
where $\Bor(\mathcal I)$ denotes the space
of all bounded Borel functions on $\mathcal I$,
as follows: 
$$
\mathcal L_Zf(x)=\sum_{\mathbf a\in Z,\, \mathbf a\rightsquigarrow b}
f(\gamma_{\mathbf a'}(x))w_{\mathbf a'}(x),\quad x\in I_b.
$$
Here the weight $w_{\mathbf a'}(x)$ is defined in~\eqref{e:w-a}.
The Patterson--Sullivan measure is invariant under the adjoint of
$\mathcal L_Z$:
\begin{lemm}
  \label{l:transfer}
Assume that $Z\subset\mathcal W^\circ$ is a partition.
Then we have for all $f\in\Bor(\mathcal I)$,
\begin{equation}
  \label{e:transfer-invariant}
\int_{\Lambda_\Gamma} f\,d\mu=\int_{\Lambda_\Gamma} \mathcal L_Z  f\,d\mu.
\end{equation}
\end{lemm}
\begin{proof}
Since $Z$ is a partition, we have by~\eqref{e:partition-intervals}
$$
\int_{\Lambda_\Gamma} f\,d\mu
=\sum_{b\in\mathcal A}\ \sum_{\mathbf a\in Z,\
\mathbf a\rightsquigarrow b} \int_{I_{\mathbf a}} f\,d\mu
$$
which together with~\eqref{e:psin-2} gives~\eqref{e:transfer-invariant}.
\end{proof}
We will use the following corollary of Lemma~\ref{l:transfer}:
\begin{equation}
  \label{e:transfer-invariant-2}
\int_{\Lambda_\Gamma}f\,d\mu=\int_{\Lambda_\Gamma} \mathcal L_Z^k f\,d\mu,\quad
f\in \Bor(\mathcal I),\
k\in\mathbb N.
\end{equation}
Note that $\mathcal L_Z^k f$ is given by the formula
\begin{equation}
  \label{e:transfer-power-1}
\mathcal L_Z^k f(x)=\sum_{\mathbf a_1,\dots,\mathbf a_k\in Z\atop
\mathbf a_1\rightsquigarrow\dots\rightsquigarrow \mathbf a_k\rightsquigarrow b}
f(\gamma_{\mathbf a_1'\dots\mathbf a_k'}(x))w_{\mathbf a_1'\dots\mathbf a_k'}(x),\quad
x\in I_b.
\end{equation}

\section{Fourier decay bound}
  \label{s:fourier}

\subsection{Key combinatorial tool}

The key tool in the proof of Theorem~\ref{t:fourier} is
the following result~\cite[Lemma~8.43]{SumProduct}
(more precisely, its version in Proposition~\ref{l:key-tool} below):
\begin{prop}
  \label{l:bourgain-literally}
For all $\delta_1>0$, there exist $\varepsilon_3,\varepsilon_4>0$ and
$k\in\mathbb N$ such that the following holds. Let $\mu_0$ be a probability
measure on $[{1\over 2},1]$ and let $N$ be a large integer. Assume that
for all $\sigma\in [N^{-1},N^{-\varepsilon_3}]$
\begin{equation}
  \label{e:bl-1}
\sup_x \mu_0\big([x-\sigma,x+\sigma]\big)< \sigma^{\delta_1}.
\end{equation}
Then for all $\eta\in \mathbb R$, $|\eta|\sim N$,
\begin{equation}
  \label{e:bl-2}
\bigg|\int \exp(2\pi i \eta x_1\cdots x_k)\,d\mu_0(x_1)\dots d\mu_0(x_k)\bigg|
\leq N^{-\varepsilon_4}.
\end{equation}
\end{prop}
\Remark
The main component of the proof of~\cite[Lemma~8.43]{SumProduct} is the discretized sum-product theorem~\cite[Theorem~1]{SumProduct}.
Roughly speaking it states that
for a finite set $A\subset [{1\over 2},1]$ of ${1\over N}$-separated points which has box dimension
$\geq\delta_1>0$, either the sum set $A+A$
or the product set $A\cdot A$ has size at least $N^{\varepsilon}\cdot \#(A)$,
where $\varepsilon>0$ depends only on~$\delta_1$.
The box dimension condition is analogous to~\eqref{e:bl-1}.
We refer the reader to the papers by the first author~\cite{Bourgain0,SumProduct}
for history and applications of the sum-product theorem.
For the passage from the sum-product theorem to the estimate~\eqref{e:bl-2}
in the cleaner case of prime fields see Bourgain--Glibichuk--Konyagin~\cite[Theorem~5]{BGK}.
See also the expository article of Green~\cite{BenGreen}.

The following is an adaptation of Proposition~\ref{l:bourgain-literally}
to the case of several different measures with slightly relaxed assumptions:
\begin{prop}
  \label{l:bourger}
Fix $\delta_0>0$. Then there exist $k\in\mathbb N$, $\varepsilon_2>0$ depending
only on~$\delta_0$ such that the following holds. Let $C_0>0$ and
$\mu_1,\dots,\mu_k$ be Borel measures on $[C_0^{-1},C_0]\subset\mathbb R$
such that $\mu_j(\mathbb R)\leq C_0$. Let $\eta\in\mathbb R$,
$|\eta|\geq 1$, and assume that
for all $\sigma\in \big[C_0|\eta|^{-1},C_0^{-1}|\eta|^{-\varepsilon_2}\big]$
and $j=1,\dots,k$
\begin{equation}
  \label{e:bourger-1}
\mu_j\times\mu_j\big(\big\{(x,y)\in \mathbb R^2\colon |x-y|\leq\sigma\big\}\big)\leq C_0\cdot \sigma^{\delta_0}.
\end{equation}
Then there exists a constant $C_1$ depending only on $C_0,\delta_0$ such that
\begin{equation}
  \label{e:bourger-2}
\bigg|\int \exp(2\pi i\eta x_1\cdots x_k)\,d\mu_1(x_1)\dots d\mu_k(x_k)\bigg|\leq
C_1|\eta|^{-\varepsilon_2}.
\end{equation}
\end{prop}
\begin{proof}
We may assume
that $|\eta|$ is large depending on $C_0,\delta_0$.
By breaking $\mu_j$ into pieces supported on $[2^\ell,2^{\ell+1}]$
where $|\ell|\lesssim \log_2 C_0$ and rescaling $\eta$, we reduce to the case when each $\mu_j$ is supported on $[{1\over 2},1]$.

Put $\delta_1:=\delta_0/6$, choose $\varepsilon_3,\varepsilon_4,k$
as in Proposition~\ref{l:bourgain-literally}, and put
$$
\varepsilon_2:={\min(\varepsilon_4,\varepsilon_3\delta_0)\over 10}.
$$
We henceforth replace~\eqref{e:bourger-1}
with the following assumption:
\begin{equation}
  \label{e:bourger-3}
\sup_x \mu_j\big([x-\sigma,x+\sigma]\big)\leq 2\sqrt{C_0}\cdot \sigma^{\delta_0/2},\quad
\sigma\in \big[C_0|\eta|^{-1},(2C_0)^{-1}|\eta|^{-\varepsilon_2}\big]
\end{equation}
which follows from~\eqref{e:bourger-1} since $[x-\sigma,x+\sigma]^2\subset \{(x,y)\in\mathbb R^2\colon |x-y|\leq 2\sigma\}$.

We next claim that it suffices to consider the case $\mu_1=\dots=\mu_k$.
Indeed, denote
$$
F(\mu_1,\dots,\mu_k):=\int \exp(2\pi i\eta x_1\cdots x_k)\,d\mu_1(x_1)\dots d\mu_k(x_k).
$$
For $\lambda:=(\lambda_1,\dots,\lambda_k)\in [0,1]^k$, put
$$
G(\lambda):=F(\mu_\lambda,\dots,\mu_\lambda),\quad
\mu_\lambda:=
\lambda_1\mu_1+\dots+\lambda_k\mu_k.
$$
If $\mu_1,\dots,\mu_k$ satisfy~\eqref{e:bourger-3},
then
the measure $\mu_\lambda$
satisfies~\eqref{e:bourger-3} as well
(with $C_0$ replaced by $k^2C_0$).
Then the version of Proposition~\ref{l:bourger} for the case $\mu_1=\dots=\mu_k$
implies that for some $C'_1$ depending only on $\delta_0,C_0$
$$
\sup_{\lambda\in [0,1]^k}|G(\lambda)|\leq C'_1|\eta|^{-\varepsilon_2}.
$$
Since $G$ is a polynomial of degree $k$,
we have for some $C_1$ depending only on $\delta_0,C_0$
$$
|F(\mu_1,\dots,\mu_k)|={1\over k!}|\partial_{\lambda_1}\dots \partial_{\lambda_k} G(0,\dots,0)|
\leq C_1|\eta|^{-\varepsilon_2}
$$
giving~\eqref{e:bourger-2} in the general case.

We henceforth assume that $\mu_1=\dots=\mu_k$. We consider two cases:
\begin{enumerate}
\item $\mu_1(\mathbb R)\geq |\eta|^{-\varepsilon_3\delta_0/10}$:
define the probability measure $\mu_0$ on $[{1\over 2},1]$ by
$$
\mu_0:={\mu_1\over\mu_1(\mathbb R)}.
$$
Choose an integer $N$ such that $N\leq |\eta|\leq 2N$.
By~\eqref{e:bourger-3} we have
$$
\sup_x \mu_0\big([x-\sigma,x+\sigma]\big) < \sigma^{\delta_1},\quad
\sigma\in [C_0N^{-1},N^{-\varepsilon_3}].
$$
Same is true for $\sigma\in [N^{-1},C_0N^{-1}]$
by applying~\eqref{e:bourger-3} to $\sigma:=C_0N^{-1}$.
Then~\eqref{e:bourger-2} follows from Proposition~\ref{l:bourgain-literally}.
\item $\mu_1(\mathbb R)\leq |\eta|^{-\varepsilon_3\delta_0/10}$:
the bound~\eqref{e:bourger-2} follows from the triangle inequality.\qedhere
\end{enumerate}
\end{proof}
In the discrete probability case Proposition~\ref{l:bourger}
gives the following statement which is used in the key step
of the proof of Theorem~\ref{t:fourier} at the end of~\S\ref{s:fourier-2}:
\begin{prop}
  \label{l:key-tool}
Fix $\delta_0>0$. Then there exist $k\in\mathbb N$, $\varepsilon_2>0$
depending only on $\delta_0$ such that
the following holds.
Let $C_0,N_{\mathcal Z}\geq 0$
and $\mathcal Z_1,\dots,\mathcal Z_k$ be finite sets such that
$\#(\mathcal Z_j)\leq C_0N_{\mathcal Z}$.
Take some maps
$$
\zeta_j:\mathcal Z_j\to [C_0^{-1},C_0],\quad
j=1,\dots,k.
$$
Let $\eta\in\mathbb R$, $|\eta|>1$, and consider the sum
$$
S_{k}(\eta)=N_{\mathcal Z}^{-k}\sum_{\mathbf b_1\in\mathcal Z_1,\dots,\mathbf b_k\in\mathcal Z_k} \exp\big(2\pi i\eta \zeta_1(\mathbf b_1)\cdots \zeta_k(\mathbf b_k)\big).
$$
Assume that $\zeta_j$ satisfy for all
$\sigma\in \big[|\eta|^{-1}, |\eta|^{-\varepsilon_2}\big]$
and $j=1,\dots,k$
\begin{equation}
  \label{e:boxdim}
\#\{(\mathbf b,\mathbf c)\in \mathcal Z_j^2\colon |\zeta_j(\mathbf b)-\zeta_j(\mathbf c)|\leq \sigma\}\leq C_0N_{\mathcal Z}^2\cdot \sigma^{\delta_0}.
\end{equation}
Then we have for some constant $C_1$ depending only on $C_0,\delta_0$
\begin{equation}
  \label{e:fourdec}
|S_{k}(\eta)|\leq C_1|\eta|^{-\varepsilon_2}.
\end{equation}
\end{prop}
\begin{proof}
It suffices to apply Proposition~\ref{l:bourger}
to the measures $\mu_j$ defined by
$$
\mu_j(A):=N_{\mathcal Z}^{-1}\cdot \#\{\mathbf b\in\mathcal Z_j\colon
\zeta_j(\mathbf b)\in A\},\quad
j=1,\dots,k.\qedhere
$$
\end{proof}

\subsection{A combinatorial description of the oscillatory integral}
  \label{s:fourier-1}

We now begin the proof of Theorem~\ref{t:fourier}.
We fix a Schottky representation for $M$ as in~\S\ref{s:schottky}.
In this section $C$ denotes constants which depend only on $C_{\varphi, g}$
and the Schottky data.

Put $\delta_0:=\delta/4$ and choose
$k\in\mathbb N$, $\varepsilon_2>0$ from Proposition~\ref{l:key-tool},
depending only on~$\delta$.
Let $\xi$ be the frequency parameter in~\eqref{e:fourier}.
Without loss
of generality we may assume that $|\xi|\geq C$.
Define the small number $\tau>0$ by
\begin{equation}
  \label{e:tau-def}
|\xi|=\tau^{-2k-3/2}.
\end{equation}
Let $Z(\tau)\subset\mathcal W^\circ$ be the partition defined in~\eqref{e:Z-tau}
and $\mathcal L_{Z(\tau)}$ be the associated transfer operator,
see~\S\ref{s:transfer}. Recall from~\eqref{e:partition-count} that
\begin{equation}
  \label{e:Z-tau-cnt}
\#(Z(\tau))\leq C\tau^{-\delta}.
\end{equation}
Moreover, by~\eqref{e:Z-tau-prop-2} and~\eqref{e:w-a}
we have for each $\mathbf a=a_1\dots a_n\in Z(\tau)$,
\begin{equation}
  \label{e:weightless}
w_{\mathbf a'}\leq C\tau^\delta\quad\text{on }I_{a_n}.
\end{equation}
We introduce some notation used throughout this section:
\begin{itemize}
\item we denote
$$
\mathbf A=(\mathbf a_0,\dots,\mathbf a_k)\in Z(\tau)^{k+1},\quad
\mathbf B=(\mathbf b_1,\dots,\mathbf b_k)\in Z(\tau)^k;
$$
\item we write $\mathbf A\leftrightarrow \mathbf B$ if and only if
$\mathbf a_{j-1}\rightsquigarrow \mathbf b_j\rightsquigarrow \mathbf a_j$
for all $j=1,\dots,k$;
\item if $\mathbf A\leftrightarrow \mathbf B$, then we define
the words $\mathbf A *\mathbf B:= \mathbf a_0' \mathbf b_1'\mathbf a_1'\mathbf b_2'\dots
\mathbf a_{k-1}'\mathbf b_{k}'\mathbf a'_k$
and $\mathbf A\#\mathbf B:=\mathbf a_0' \mathbf b_1'\mathbf a_1'\mathbf b_2'\dots
\mathbf a_{k-1}'\mathbf b_{k}'$;
\item denote by $b(\mathbf A)\in\mathcal A$ the last letter of $\mathbf a_k$;
\item for each $\mathbf a\in\mathcal W^\circ$, denote by $x_{\mathbf a}$
the center of $I_{\mathbf a}$;
\item for $j\in \{1,\dots,k\}$ and $\mathbf b\in Z(\tau)$
such that $\mathbf a_{j-1}\rightsquigarrow \mathbf b\rightsquigarrow\mathbf a_j$,
define
\begin{equation}
  \label{e:zeta-def}
\zeta_{j,\mathbf A}(\mathbf b):=\tau^{-2}\gamma'_{\mathbf a_{j-1}'\mathbf b'}(x_{\mathbf a_j})
\end{equation}
and note that $\zeta_{j,\mathbf A}(\mathbf b)\in [C^{-1},C]$
by the chain rule and~\eqref{e:Z-tau-prop-2}.
\end{itemize}
Using the functions $\varphi,g$ from the
statement of Theorem~\ref{t:fourier}, define
\begin{equation}
  \label{e:the-f}
f(x):=\exp(i\xi\varphi(x))g(x),\quad
x\in\Lambda_\Gamma.
\end{equation}
By~\eqref{e:transfer-invariant-2} and~\eqref{e:transfer-power-1}
the integral in~\eqref{e:fourier}
can be written as follows:
\begin{equation}
  \label{e:massage-2}
\int_{\Lambda_\Gamma}f\,d\mu
=\int_{\Lambda_\Gamma}\mathcal L_{Z(\tau)}^{2k+1}f\,d\mu
=\sum_{\mathbf A,\mathbf B\colon\mathbf A\leftrightarrow \mathbf B}
\int_{I_{b(\mathbf A)}}f(\gamma_{\mathbf A*\mathbf B}(x))w_{\mathbf A*\mathbf B}(x)\,d\mu(x).
\end{equation}
We use H\"older's inequality
and approximations for
the weight $w_{\mathbf A*\mathbf B}$ and the amplitude $g$ to get the following bound.
Note that~\eqref{e:Z-tau-prop-2} and~\eqref{e:tau-def} imply
that the function $e^{i\xi\varphi(\gamma_{\mathbf A*\mathbf B}(x))}$
below oscillates at frequencies $\sim \tau^{-1/2}$.
\begin{lemm}
  \label{l:approx1}
We have
\begin{equation}
  \label{e:approx1}
\bigg|\int_{\Lambda_\Gamma}f\,d\mu\bigg|^2\leq
C\tau^{(2k-1)\delta}
\sum_{\mathbf A,\mathbf B\colon\mathbf A\leftrightarrow \mathbf B}
\bigg|\int_{I_{b(\mathbf A)}} e^{i\xi\varphi(\gamma_{\mathbf A*\mathbf B}(x))}w_{\mathbf a_k'}(x)\,d\mu(x)\bigg|^2+
C\tau^2.
\end{equation}
\end{lemm}
\begin{proof}
Take arbitrary $x\in I_{b(\mathbf A)}$, then
$$
w_{\mathbf A*\mathbf B}(x)=w_{\mathbf A\#\mathbf B}(\gamma_{\mathbf a'_k}(x))
w_{\mathbf a'_k}(x).
$$
Now, $\gamma_{\mathbf a'_k}(x)$ lies in $I_{\mathbf a_k}$,
which by~\eqref{e:Z-tau} is an interval of size no more
than $\tau$.
By~\eqref{e:tree-derivative-2}
\begin{equation}
  \label{e:koala-1}
\exp(-C\tau)\leq {w_{\mathbf A\#\mathbf B}(\gamma_{\mathbf a'_k}(x))
\over w_{\mathbf A\#\mathbf B}(x_{\mathbf a_k})}\leq \exp(C\tau).
\end{equation}
Moreover, by~\eqref{e:weightless} and the chain rule
\begin{equation}
  \label{e:koala-2}
w_{\mathbf A^*\mathbf B}(x)\leq C\tau^{(2k+1)\delta},\quad
w_{\mathbf A\#\mathbf B}(x_{\mathbf a_k})\leq C\tau^{2k\delta}.
\end{equation}
Recall that $\|g\|_{C^1}\leq C$ by~\eqref{e:fourier-hyp-2}.
Since $\gamma_{\mathbf A*\mathbf B}(x)\in I_{\mathbf a_0}$, by~\eqref{e:Z-tau} we have
\begin{equation}
  \label{e:koala-3}
|f(\gamma_{\mathbf A*\mathbf B}(x))-e^{i\xi\varphi(\gamma_{\mathbf A*\mathbf B}(x))}g(x_{\mathbf a_0})|\leq C\tau.
\end{equation}
Put 
$$
g_{\mathbf A,\mathbf B}:=w_{\mathbf A\#\mathbf B}(x_{\mathbf a_k})
g(x_{\mathbf a_0}).
$$
Combining~\eqref{e:koala-1}--\eqref{e:koala-3}, we obtain
$$
|f(\gamma_{\mathbf A*\mathbf B}(x))w_{\mathbf A*\mathbf B}(x)-g_{\mathbf A,\mathbf B}
e^{i\xi\varphi(\gamma_{\mathbf A*\mathbf B}(x))}w_{\mathbf a_k'}(x)|\leq C\tau^{(2k+1)\delta+1}.
$$
Therefore by~\eqref{e:massage-2} and~\eqref{e:Z-tau-cnt}
\begin{equation}
  \label{e:panda-1}
\bigg|\int_{\Lambda_\Gamma}f\,d\mu-\sum_{\mathbf A,\mathbf B\colon\mathbf A\leftrightarrow \mathbf B}
g_{\mathbf A,\mathbf B}\int_{I_{b(\mathbf A)}}e^{i\xi\varphi(\gamma_{\mathbf A*\mathbf B}(x))}
w_{\mathbf a'_k}(x)\,d\mu(x)
\bigg|\leq C\tau.
\end{equation}
Using H\"older's inequality, \eqref{e:Z-tau-cnt}, and~\eqref{e:koala-2}, we get
\begin{equation}
  \label{e:panda-2}
\begin{gathered}
\bigg|\sum_{\mathbf A,\mathbf B\colon\mathbf A\leftrightarrow \mathbf B}
g_{\mathbf A,\mathbf B}\int_{I_{b(\mathbf A)}}e^{i\xi\varphi(\gamma_{\mathbf A*\mathbf B}(x))}
w_{\mathbf a'_k}(x)\,d\mu(x)\bigg|^2
\\\leq
C \tau^{(2k-1)\delta}
\sum_{\mathbf A,\mathbf B\colon\mathbf A\leftrightarrow \mathbf B}
\bigg|\int_{I_{b(\mathbf A)}}e^{i\xi\varphi(\gamma_{\mathbf A*\mathbf B}(x))}
w_{\mathbf a'_k}(x)\,d\mu(x)\bigg|^2.
\end{gathered}
\end{equation}
Combining~\eqref{e:panda-1} and~\eqref{e:panda-2}
finishes the proof.
\end{proof}
To handle the first term on the right-hand side of~\eqref{e:approx1}, we estimate
using~\eqref{e:weightless}
\begin{equation}
  \label{e:square-revealed}
\begin{gathered}
\sum_{\mathbf A,\mathbf B\colon\mathbf A\leftrightarrow \mathbf B}
\bigg|\int_{I_{b(\mathbf A)}} e^{i\xi\varphi(\gamma_{\mathbf A*\mathbf B}(x))}w_{\mathbf a_k'}(x)\,d\mu(x)\bigg|^2\\
=\sum_{\mathbf A}
\int_{I_{b(\mathbf A)}^2} w_{\mathbf a'_k}(x)w_{\mathbf a'_k}(y)\sum_{\mathbf B\colon\mathbf A\leftrightarrow\mathbf B}e^{i\xi(\varphi(\gamma_{\mathbf A*\mathbf B}(x))
-\varphi(\gamma_{\mathbf A*\mathbf B}(y))
)}\,d\mu(x)d\mu(y)\\
\leq 
C\tau^{2\delta}\sum_{\mathbf A}
\int_{I_{b(\mathbf A)}^2} \bigg|\sum_{\mathbf B\colon\mathbf A\leftrightarrow\mathbf B}e^{i\xi(\varphi(\gamma_{\mathbf A*\mathbf B}(x))
-\varphi(\gamma_{\mathbf A*\mathbf B}(y))
)}\bigg|\,d\mu(x)d\mu(y).
\end{gathered}
\end{equation}
The next statement
bounds the integral $\int f\,d\mu$ by
an expression which can be analyzed using Proposition~\ref{l:key-tool},
by linearizing the phase $\varphi$.
Recall the definition~\eqref{e:zeta-def} of $\mathbf \zeta_{j,\mathbf A}(\mathbf b)$.
\begin{lemm}
  \label{l:approx2}
Denote
\begin{equation}
  \label{e:J-tau}
J_\tau:=\{\eta\in\mathbb R\colon \tau^{-1/4}\leq |\eta|\leq C\tau^{-1/2}\}
\end{equation}
where $C$ is sufficiently large.
Then
$$
\bigg|\int_{\Lambda_\Gamma}f\,d\mu\bigg|^2\leq
C\tau^{(2k+1)\delta}\sum_{\mathbf A}\sup_{\eta\in J_\tau}
\bigg|\sum_{\mathbf B\colon\mathbf A\leftrightarrow\mathbf B}e^{2\pi i\eta\zeta_{1,\mathbf A}(\mathbf b_1)
\cdots\zeta_{k,\mathbf A}(\mathbf b_k)}\bigg|+C\tau^{\delta/4}.
$$  
\end{lemm}
\begin{proof}
Fix $\mathbf A$.
Take $x,y\in I_{b(\mathbf A)}$ and put
$$
\tilde x:=\gamma_{\mathbf a'_k}(x),\
\tilde y:=\gamma_{\mathbf a'_k}(y)\ \in\ I_{\mathbf a_k}.
$$
Assume that $\mathbf A\leftrightarrow\mathbf B$.
Since $\gamma_{\mathbf A*\mathbf B}(x)=\gamma_{\mathbf A\#\mathbf B}(\tilde x)$,
$\gamma_{\mathbf A*\mathbf B}(y)=\gamma_{\mathbf A\#\mathbf B}(\tilde y)$,
we have
$$
\varphi(\gamma_{\mathbf A*\mathbf B}(y))-\varphi(\gamma_{\mathbf A*\mathbf B}(x))
=\int_{\tilde x}^{\tilde y}(\varphi\circ\gamma_{\mathbf A\#\mathbf B})'(t)\,dt.
$$
By the chain rule, for each $t\in I_{\mathbf a_k}$ there
exist $s_j\in I_{\mathbf a_j}$, $j=0,\dots,k$, such that
$$
(\varphi\circ\gamma_{\mathbf A\#\mathbf B})'(t)
=\varphi'(s_0)\gamma_{\mathbf a_0'\mathbf b_1'}'(s_1)\cdots
\gamma_{\mathbf a_{k-1}'\mathbf b_k'}'(s_k).
$$
By~\eqref{e:Z-tau}, we have $|s_j-x_{\mathbf a_j}|\leq \tau$.
Then by~\eqref{e:fourier-hyp-2} and~\eqref{e:tree-derivative-2}, we have for all 
$t\in I_{\mathbf a_k}$
$$
\exp(-C\tau)\leq{(\varphi\circ\gamma_{\mathbf A\#\mathbf B})'(t)\over
\tau^{2k}\varphi'(x_{\mathbf a_0})\zeta_{1,\mathbf A}(\mathbf b_1)
\cdots\zeta_{k,\mathbf A}(\mathbf b_k)}
\leq \exp(C\tau).
$$
Since $|\varphi'(x_{\mathbf a_0})|,\zeta_{j,\mathbf A}(\mathbf b_j)\in [C^{-1},C]$
and $|\tilde x-\tilde y|\leq \tau$,
it follows that
\begin{equation}
  \label{e:wallaby}
\big|\varphi(\gamma_{\mathbf A*\mathbf B}(x))-\varphi(\gamma_{\mathbf A*\mathbf B}(y))
-\tau^{2k}\varphi'(x_{\mathbf a_0})\zeta_{1,\mathbf A}(\mathbf b_1)
\cdots\zeta_{k,\mathbf A}(\mathbf b_k)(\tilde x-\tilde y)\big|\leq C\tau^{2k+2}.
\end{equation}
Denote
$$
\eta:={\sgn\xi\over 2\pi}\tau^{-3/2}\varphi'(x_{\mathbf a_0})\cdot (\tilde x-\tilde y)
$$
and note that by~\eqref{e:fourier-hyp-2} and~\eqref{e:Z-tau-prop-2}
$$
C^{-1}\tau^{-1/2}|x-y|\leq |\eta|\leq C\tau^{-1/2}|x-y|.
$$
We have by Lemma~\ref{l:approx1}, \eqref{e:square-revealed}, \eqref{e:wallaby},
and~\eqref{e:Z-tau-cnt},
recalling that $|\xi|=\tau^{-2k-3/2}$ by~\eqref{e:tau-def}
$$
\bigg|\int_{\Lambda_\Gamma}f\,d\mu\bigg|^2\leq
C\tau^{(2k+1)\delta}
\sum_{\mathbf A}
\int_{I_{b(\mathbf A)}^2} \bigg|\sum_{\mathbf B\colon\mathbf A\leftrightarrow\mathbf B}e^{2\pi i\eta\zeta_{1,\mathbf A}(\mathbf b_1)
\cdots\zeta_{k,\mathbf A}(\mathbf b_k)}\bigg|\,d\mu(x)d\mu(y)+C\sqrt\tau.
$$
Now, we remark that by~\eqref{e:adreg-1}, for each fixed constant $C_0$
$$
\mu\times\mu\big\{(x,y)\in \Lambda_\Gamma^2\colon |x-y|\leq C_0\tau^{1/4}\big\}
\leq C\tau^{\delta/4}.
$$
Therefore, the double integral above can be taken
over $x,y$ such that $|x-y|\geq C_0\tau^{1/4}$,
which for large enough $C_0$ implies that $\eta\in J_\tau$.
 This finishes
the proof.
\end{proof}

\subsection{End of the proof of Theorem~\ref{t:fourier}}
  \label{s:fourier-2}

To apply Proposition~\ref{l:key-tool} to the sum in Lemma~\ref{l:approx2},
we need a positive box dimension estimate. To state
it we recall the notation $\mathbf a\rightsquigarrow {\mathbf b\atop\mathbf c}
\rightsquigarrow \mathbf d$ from~\eqref{e:quads}
and the constant $\varepsilon_2$ fixed at the beginning of~\S\ref{s:fourier-1}.
\begin{lemm}
  \label{l:box-estimate}
Define the set of \textbf{regular sequences}
$\mathcal R\subset Z(\tau)^{k+1}$ as follows: $\mathbf A\in \mathcal R$
if and only if for all $j=1,\dots,k$ and $\sigma\in [\tau,\tau^{\varepsilon_2/4}]$
we have
\begin{equation}
  \label{e:regular-words}
\tau^{2\delta}\cdot\#\big\{(\mathbf b,\mathbf c)\in Z(\tau)^2\colon
\mathbf a_{j-1}\rightsquigarrow {\textstyle{\mathbf b\atop\mathbf c}}\rightsquigarrow\mathbf a_j,\
|\zeta_{j,\mathbf A}(\mathbf b)-\zeta_{j,\mathbf A}(\mathbf c)|\leq\sigma
\big\}\leq \sigma^{\delta/4}.
\end{equation}
Then most sequences are regular, more precisely 
\begin{equation}
  \label{e:irregular-words}
\tau^{(k+1)\delta}\cdot \#(Z(\tau)^{k+1}\setminus\mathcal R)\leq C\tau^{\varepsilon_2\delta/20}.
\end{equation}
\end{lemm}
\begin{proof}
For $\ell\in\mathbb Z$ with $\tau\leq 2^{-\ell}\leq 2\tau^{\varepsilon_2/4}$,
define $\widetilde{\mathcal R}_\ell$ as the set of pairs
$(\mathbf a,\mathbf d)\in Z(\tau)^2$ such that
$$
\tau^{2\delta}\cdot\#\big\{(\mathbf b,\mathbf c)\in Z(\tau)^2\colon
\mathbf a\rightsquigarrow{\textstyle{\mathbf b\atop\mathbf c}}\rightsquigarrow\mathbf d,\
|\gamma'_{\mathbf a'\mathbf b'}(x_{\mathbf d}) - \gamma'_{\mathbf a'\mathbf c'}(x_{\mathbf d})| \leq \tau^2 2^{-\ell}
\big\}\leq 2^{-(\ell+1)\delta/4}.
$$
For each $\sigma\in [\tau,\tau^{\varepsilon_2/4}]$,
there exists $\ell$ such that $2^{-\ell-1}\leq\sigma\leq 2^{-\ell}$.
By~\eqref{e:zeta-def},
$$
\bigcap_j\bigcap_\ell \big\{\mathbf A\in Z(\tau)^{k+1}\mid (\mathbf a_{j-1},\mathbf a_j)\in \widetilde{\mathcal R}_\ell\big\}\subset \mathcal R.
$$
It suffices to show that for each $j,\ell$ we have
\begin{equation}
  \label{e:dc}
\tau^{2\delta}\cdot\#(Z(\tau)^2\setminus\widetilde{\mathcal R}_\ell)\leq C\tau^{\varepsilon_2\delta/16}.
\end{equation}
By Chebyshev's inequality the left-hand side of~\eqref{e:dc} is bounded above by
$$
2^{(\ell+1)\delta/4}\tau^{4\delta}\cdot\#\{(\mathbf a,\mathbf b,\mathbf c,\mathbf d)\in Z(\tau)^4\colon
\mathbf a\rightsquigarrow{\textstyle{\mathbf b\atop\mathbf c}}\rightsquigarrow\mathbf d,\
|\gamma'_{\mathbf a'\mathbf b'}(x_{\mathbf d}) - \gamma'_{\mathbf a'\mathbf c'}(x_{\mathbf d})| \leq \tau^2 2^{-\ell}\}
$$
By Lemma~\ref{l:triple-count} this is bounded above by
$$
C2^{-\delta\ell/4}\leq C\tau^{\varepsilon_2\delta/16}.
$$
This gives~\eqref{e:dc}, finishing the proof.
\end{proof}
We are now ready to finish the proof of Theorem~\ref{t:fourier}.
Using Lemma~\ref{l:approx2}
and estimating the sum over $\mathbf A\in Z(\tau)^{k+1}\setminus \mathcal R$
by Lemma~\ref{l:box-estimate}, we obtain
\begin{equation}
  \label{e:endgame-1}
\bigg|\int_{\Lambda_\Gamma}f\,d\mu\bigg|^2
\leq C\tau^{k\delta}\max_{\mathbf A\in\mathcal R}\sup_{\eta\in J_\tau}
\bigg|\sum_{\mathbf B\colon \mathbf A\leftrightarrow\mathbf B}
e^{2\pi i\eta\zeta_{1,\mathbf A}(\mathbf b_1)\cdots\zeta_{k,\mathbf A}(\mathbf b_k)}\bigg|
+C\tau^{\varepsilon_2\delta/20}.
\end{equation}
We estimate the first term on the right-hand side using Proposition~\ref{l:key-tool}.
Fix $\mathbf A\in \mathcal R$ and define
$$
\mathcal Z_j:=\{\mathbf b\in Z(\tau)\colon
\mathbf a_{j-1}\rightsquigarrow\mathbf b\rightsquigarrow \mathbf a_j\},\quad
j=1,\dots,k.
$$
By~\eqref{e:partition-count},
$$
\#(\mathcal Z_j)\leq CN_{\mathcal Z},\quad
N_{\mathcal Z}:=\tau^{-\delta}.
$$
Fix $\eta\in J_\tau$. Recall that $\delta_0=\delta/4$.
By~\eqref{e:J-tau} and~\eqref{e:regular-words} we have
for all $j=1,\dots,k$ and $\sigma\in \big[|\eta|^{-1},|\eta|^{-\varepsilon_2}\big]$
$$
\#\big\{(\mathbf b,\mathbf c)\in \mathcal Z_j^2\colon
|\zeta_{j,\mathbf A}(\mathbf b)-\zeta_{j,\mathbf A}(\mathbf c)|\leq\sigma
\big\}\leq N_{\mathcal Z}^2\cdot\sigma^{\delta_0}.
$$
Therefore, condition~\eqref{e:boxdim} is satisfied.
We also recall from~\eqref{e:zeta-def} that
$\zeta_{j,\mathbf A}(\mathbf b)\in [C^{-1},C]$.

Applying Proposition~\ref{l:key-tool}, we obtain
for all $\mathbf A\in\mathcal R$ and $\eta\in J_\tau$
\begin{equation}
  \label{e:endgame-2}
\tau^{k\delta}\bigg|\sum_{\mathbf B\colon\mathbf A\leftrightarrow\mathbf B}e^{2\pi i\eta\zeta_{1,\mathbf A}(\mathbf b_1)
\cdots\zeta_{k,\mathbf A}(\mathbf b_k)}\bigg|
\leq C|\eta|^{-\varepsilon_2}
\leq C\tau^{\varepsilon_2/4}.
\end{equation}
From~\eqref{e:endgame-1} and \eqref{e:endgame-2} we have
$$
\bigg|\int_{\Lambda_\Gamma}f\,d\mu\bigg|\leq C\tau^{\varepsilon_2\delta/40}.
$$
Recalling~\eqref{e:tau-def}
and the definition~\eqref{e:the-f} of $f$,
this gives Theorem~\ref{t:fourier} with
\begin{equation}
  \label{e:epsilon-1-revealed}
\varepsilon_1:={\varepsilon_2\delta\over 40(2k+3/2)}.
\end{equation}

\section{Fractal uncertainty principle}
  \label{s:fup}

In this section, we deduce Theorem~\ref{t:gap} from Theorem~\ref{t:fourier}
by establishing a fractal uncertainty principle (henceforth denoted
FUP) and using the results of~\cite{hgap}.
Throughout this section we assume that $M,\delta,\Lambda_\Gamma,\mu$
are as in Theorem~\ref{t:fourier}.

\subsection{FUP for the Patterson--Sullivan measure}

We first use Theorem~\ref{t:fourier} to obtain a fractal uncertainty
principle with respect to the Patterson--Sullivan measure $\mu$:
\begin{prop}
  \label{l:fup-measure}
Assume that:
\begin{itemize}
\item $U\subset\mathbb R^2$ is an open set and $V\subset U$ is compact;
\item $\Phi\in C^3(U;\mathbb R)$ and $G\in C^1(U;\mathbb C)$, $\supp G\subset V$,
satisfy for some constant $C_{\Phi,G}$
\begin{equation}
  \label{e:fup-measure-0}
\|\Phi\|_{C^3}+\|G\|_{C^1}\leq C_{\Phi, G},\quad
\inf |\partial^2_{xy}\Phi|\geq C_{\Phi, G}^{-1}.
\end{equation}
\end{itemize}
Define for $0<h<1$ the operator $B(h):L^2(\Lambda_\Gamma;\mu)\to L^2(\Lambda_\Gamma;\mu)$ by
\begin{equation}
  \label{e:B-h-def}
B(h) u(x)=\int_{\Lambda_\Gamma} \exp\Big({i\Phi(x,y)\over h}\Big)G(x,y)u(y)\,d\mu(y).
\end{equation}
Let $\varepsilon_1=\varepsilon_1(\delta)>0$ be the constant from Theorem~\ref{t:fourier}. Then
\begin{equation}
  \label{e:fup-measure}
\|B(h)\|_{L^2(\Lambda_\Gamma;\mu)\to L^2(\Lambda_\Gamma;\mu)}\leq
Ch^{\varepsilon_1/4},\quad
0<h<1
\end{equation}
where the constant $C$ depends only on $M,U,V,C_{\Phi, G}$.
\end{prop}
\begin{proof}
We denote by $C$ constants which depend only on
$M,U,V,C_{\Phi, G}$. As in~\S\ref{s:schottky},
we view $\Lambda_\Gamma$ as a subset of $\mathbb R$.
Using a partition of unity for $G$, we reduce to the case
$$
U=I_1^\circ\times I_2^\circ,\quad
V=J_1\times J_2,\quad
J_1\subset I_1^\circ,\quad
J_2\subset I_2^\circ
$$
for some intervals $I_1,I_2,J_1,J_2$.
To prove~\eqref{e:fup-measure} suffices to show that
\begin{equation}
  \label{e:fup-measure-1}
\|B(h)B(h)^*\|_{L^2(\Lambda_\Gamma;\mu)\to L^2(\Lambda_\Gamma;\mu)}\leq Ch^{\varepsilon_1/2}.
\end{equation}
Note that $B(h)B(h)^*$ is an integral operator:
$$
B(h)B(h)^* f(x)=\int_{\Lambda_\Gamma} \mathcal K(x,x')f(x')\,d\mu(x'),
$$
where
$$
\mathcal K(x,x')=\int_{\Lambda_\Gamma}\exp\Big({i\over h}\big(\Phi(x,y)-\Phi(x',y)\big)\Big)
G(x,y)\overline{G(x',y)}\,d\mu(y).
$$
By Schur's inequality, to show~\eqref{e:fup-measure-1} it suffices to prove the bound
\begin{equation}
  \label{e:schur}
\sup_{x\in\Lambda_\Gamma}
\int_{\Lambda_\Gamma} |\mathcal K(x,x')|\,d\mu(x')\leq Ch^{\varepsilon_1/2}.
\end{equation}
For $x,x'\in\Lambda_\Gamma\cap J_1$, define the functions
$\varphi_{xx'}$, $g_{xx'}$ on $I_2^\circ$ as follows:
$$
\Phi(x,y)-\Phi(x',y)=(x-x')\cdot \varphi_{xx'}(y),\quad
g_{xx'}(y)=G(x,y)\overline{G(x',y)}.
$$
Then
\begin{equation}
  \label{e:fup-measure-2}
\mathcal K(x,x')
=\int_{\Lambda_\Gamma} \exp\big(i\xi\varphi_{xx'}(y)\big)g_{xx'}(y)\,d\mu(y),\quad
\xi:={x-x'\over h}.
\end{equation}
It follows from~\eqref{e:fup-measure-0} that
$$
\|\varphi_{xx'}\|_{C^2(I_2^\circ)}+
\|g_{xx'}\|_{C^1(I_2^\circ)}\leq C,\quad
\inf_{I_2^\circ} |\partial_y\varphi_{xx'}|\geq C^{-1}
$$
and we extend $g_{xx'},\varphi_{xx'}$ to compactly supported functions on $\mathbb R$ so that
$$
\|\varphi_{xx'}\|_{C^2(\mathbb R)}+
\|g_{xx'}\|_{C^1(\mathbb R)}
\leq C,\quad
\inf_{\Lambda_\Gamma} |\partial_y\varphi_{xx'}|\geq C^{-1};
$$
this is possible since $\Lambda_\Gamma\subset\mathbb R$ is compact.

Applying Theorem~\ref{t:fourier} and using~\eqref{e:fup-measure-2} we get the bound
\begin{equation}
  \label{e:fourier-used}
|\mathcal K(x,x')|\leq C\Big|{x-x'\over h}\Big|^{-\varepsilon_1},\quad
x,x'\in \Lambda_\Gamma\cap J_1,\quad
|x-x'|\geq h.
\end{equation}
It remains to split the integral in~\eqref{e:schur} into two parts.
The integral over $\{|x-x'|\leq h^{1/2}\}$ is
bounded by $Ch^{\delta/2}$ by~\eqref{e:adreg-1}.
The integral over $\{|x-x'|\geq h^{1/2}\}$ is
bounded by $Ch^{\varepsilon_1/2}$ by~\eqref{e:fourier-used}.
\end{proof}

\subsection{FUP for the Lebesgue measure}

We now deduce from Proposition~\ref{l:fup-measure}
a fractal uncertainty principle with respect
to Lebesgue measure on a neighborhood
$$
\Lambda_\Gamma(h):=\Lambda_\Gamma+[-h,h]\subset\mathbb R
$$
of $\Lambda_\Gamma$. We use the following
\begin{lemm}
  \label{l:measure-convolved}
For $0<h<1$, define the function $F_h(x)$ as the convolution
of the Patterson--Sullivan measure $\mu$ with the rescaled
uniform measure on $[-2h,2h]$:
\begin{equation}
  \label{e:F-h}
F_h(x):={1\over 4h^\delta}\mu\big([x-2h,x+2h]\big).
\end{equation}
Then for some constant $C_\Gamma>0$ depending only on $\Gamma$,
\begin{equation}
  \label{e:F-h-bdd}
F_h\geq C_\Gamma^{-1}\quad\text{on }\Lambda_\Gamma(h).
\end{equation}
\end{lemm}
\begin{proof}
Let $x\in \Lambda_\Gamma(h)$. Then there exists
$x_0\in\Lambda_\Gamma$ such that $|x-x_0|\leq h$.
We have $[x_0-h,x_0+h]\subset [x-2h,x+2h]$
and $\mu([x_0-h,x_0+h])\geq C_\Gamma^{-1}h^\delta$
by~\eqref{e:adreg-2}.
Therefore $F_h(x)\geq C_\Gamma^{-1}$.
\end{proof}
Our fractal uncertainty principle for the Lebesgue measure is the following
\begin{prop}
  \label{l:fup}
Let $\varepsilon_1=\varepsilon_1(\delta)>0$ be the constant from Theorem~\ref{t:fourier}.
Assume that $U,V,\Phi,G$ are as in Proposition~\ref{l:fup-measure}.
Define the operator $\mathcal B(h):L^2(\mathbb R)\to L^2(\mathbb R)$ by
\begin{equation}
  \label{e:B-h-def-2}
\mathcal B(h)u(x)=(2\pi h)^{-1/2}\int_{\mathbb R}\exp\Big({i\Phi(x,y)\over h}\Big)G(x,y)u(y)\,dy.
\end{equation}
Fix $\rho\in (0,1)$. Then
\begin{equation}
  \label{e:fup}
\|\indic_{\Lambda_\Gamma(h^\rho)}\mathcal B(h)\indic_{\Lambda_\Gamma(h^\rho)}\|_{L^2(\mathbb R)\to L^2(\mathbb R)}
\leq Ch^{\beta-(1-\delta)(1-\rho)},\quad
\beta:={1\over 2}-\delta+{\varepsilon_1\over 4}. 
\end{equation}
\end{prop}
\begin{proof}
Let $F_{h^\rho}$ be the function defined in~\eqref{e:F-h}, with $h$ replaced by $h^\rho$.
By~\eqref{e:F-h-bdd}, it is enough to show the following estimate
for each bounded Borel function $u$ on $\mathbb R$:
\begin{equation}
  \label{e:fup-int-1}
\|\sqrt{F_{h^\rho}}\mathcal B(h) F_{h^\rho}u\|_{L^2(\mathbb R)}\leq Ch^{\beta-(1-\delta)(1-\rho)}\|\sqrt{F_{h^\rho}}\,u\|_{L^2(\mathbb R)}.
\end{equation}
Define the shift operator $\omega_t$ on functions on $\mathbb R$ by
$$
\omega_t v(x)=v(x-t),\quad
t,x\in\mathbb R.
$$
Then for each bounded Borel function $v$ on $\mathbb R$,
$$
\|\sqrt{F_{h^\rho}}\,v\|^2_{L^2(\mathbb R)}={1\over 4h^{\rho\delta}}
\int_{-2h^\rho}^{2h^\rho}\|\omega_t v\|_{L^2(\Lambda_\Gamma;\mu)}^2\,dt.
$$
Moreover
$$
\omega_t\mathcal B(h)F_{h^\rho}u={1\over 4\sqrt{2\pi}h^{1/2+\rho\delta}}
\int_{-2h^\rho}^{2h^\rho} B_{ts}(h)\omega_s u\,ds
$$
where
$$
B_{ts}(h)v(x)=\int_{\Lambda_\Gamma} \exp\Big({i\Phi(x-t,y-s)\over h}\Big)G(x-t,y-s)v(y)\,d\mu(y).
$$
By Proposition~\ref{l:fup-measure}, we have for all $t,s\in [-2h^{\rho},2h^{\rho}]$,
$$
\|B_{ts}(h)\|_{L^2(\Lambda_\Gamma;\mu)\to L^2(\Lambda_\Gamma;\mu)}\leq Ch^{\varepsilon_1/4}.
$$
Then
$$
\begin{aligned}
\|\sqrt{F_{h^\rho}}\mathcal B(h) F_{h^\rho}u\|_{L^2(\mathbb R)}^2 &
={1\over 4h^{\rho\delta}}
\int_{-2h^\rho}^{2h^\rho}\|\omega_t\mathcal B(h)F_{h^\rho}u\|_{L^2(\Lambda_\Gamma;\mu)}^2\,dt
\\&\leq
h^{2\rho-3\rho\delta-1}\sup_{|t|\leq 2h^\rho}
\int_{-2h^\rho}^{2h^\rho}\|B_{ts}(h)\omega_s u\|_{L^2(\Lambda_\Gamma;\mu)}^2\,ds
\\&\leq
Ch^{2\rho-3\rho\delta-1+\varepsilon_1/2}
\int_{-2h^\rho}^{2h^\rho}\|\omega_s u\|_{L^2(\Lambda_\Gamma;\mu)}^2\,ds
\\&=
4Ch^{2\rho-2\rho\delta-1+\varepsilon_1/2}\|\sqrt{F_{h^\rho}}\,u\|_{L^2(\mathbb R)}^2
\end{aligned}
$$
which gives~\eqref{e:fup-int-1}.
\end{proof}

\subsection{Proof of Theorem~\ref{t:gap}}

We use~\cite[Theorem~3]{hgap}. It suffices to show that $\Lambda_\Gamma$
satisfies the fractal uncertainty principle with exponent $\beta={1\over 2}-\delta+{\varepsilon_1\over 4}$
in the sense of~\cite[Definition~1.1]{hgap}.

The paper~\cite{hgap} uses the Poincar\'e disk model of the hyperbolic plane and the limit
set there is a subset of the circle $\mathbb S^1\subset\mathbb C$. To relate to our model,
we use the standard transformation from the upper half-plane model to the disk model,
\begin{equation}
  \label{e:transform}
z\mapsto w={z-i\over z+i}.
\end{equation}
Note that, with $|\bullet|$ denoting the Euclidean norm on $\mathbb C$, we have
for $x,y\in\mathbb R$
$$
|w(x)-w(y)|^2={4(x-y)^2\over (1+x^2)(1+y^2)}.
$$
Let $\chi\in C^\infty(\mathbb S^1\times\mathbb S^1)$ satisfy
$\supp\chi\cap \{w=w'\}=\emptyset$, and $\mathcal B_\chi(h)$ be the operator
defined in~\cite[(1.6)]{hgap}. For the purpose of satisfying~\cite[Definition~1.1]{hgap}
we may assume that $\chi$ is supported near $\Lambda_\Gamma^2$, in particular
the pullback of $\chi$ to $\mathbb R^2$ by the square of the map~\eqref{e:transform}
is supported in a compact subset of
$\{(x,y)\in \mathbb R^2\mid x\neq y\}$.
Then the operator $\mathcal B_\chi(h)$
has the form~\eqref{e:B-h-def-2} with
$$
U\Subset\{(x,y)\in\mathbb R^2\mid x\neq y\},\quad
\Phi(x,y)=2\log|x-y|-\log(1+x^2)-\log(1+y^2),
$$
and we have on $U$,
$$
\partial^2_{xy}\Phi(x,y)={2\over (x-y)^2}\neq 0.
$$
It remains to apply Proposition~\ref{l:fup} to see that the fractal uncertainty principle~\cite[Definition~1.1]{hgap}
holds, finishing the proof.

\medskip\noindent\textbf{Acknowledgements.}
JB is partially supported by NSF grant DMS-1301619.
This research was conducted during the period SD served as
a Clay Research Fellow.


\end{document}